\documentclass[11pt]{amsart}
\usepackage{geometry}
\usepackage{amssymb}
\usepackage{amsfonts}
\usepackage{amsthm}
\usepackage{enumerate}
\usepackage{color}
\usepackage{mathrsfs}

\newtheorem{theorem}{Theorem}[section]
\newtheorem{lemma}[theorem]{Lemma}
\newtheorem{proposition}[theorem]{Proposition}
\newtheorem{corollary}[theorem]{Corollary}

\newtheorem{claim}{Claim}
\newtheorem{definition}[theorem]{Definition}

\newtheorem{notation}[theorem]{Notation}
\newtheorem{assumption}[theorem]{Assumption}

\newtheorem*{theorem:a}{Theorem \ref{thm:a}}

\newtheorem{remark}[theorem]{Remark}

\newcommand{\T}{\mathcal T}

\newcommand{\M}{\mathcal M}
\newcommand{\A}{\mathcal A}
\newcommand{\N}{\mathbb N}
\newcommand{\B}{\mathfrak B_{\text{mT}}}
\newcommand{\bb}{\mathscr B_{mT}}
\newcommand{\mT}{T}
\newcommand{\ah}{\mathfrak X_{\text{AH}}}
\newcommand{\am}{\mathfrak X_{\text{nr}}}
\newcommand{\mps}{\mathfrak X_{\text{Kus}}}
\newcommand{\gpz}{\mathfrak X_{p}}
\newcommand{\Li}{\mathscr L_\infty}

\DeclareMathOperator\w{w}
\DeclareMathOperator\age{age}
\DeclareMathOperator\mt{mt}

\DeclareMathOperator\bdp{bd}

\DeclareMathOperator\rank{rank}
\DeclareMathOperator\rng{rng}
\DeclareMathOperator\supp{supp}
\DeclareMathOperator\len{len}

\title[Regular subspaces of the space $\B$]{Regular subspaces of a Bourgain-Delbaen space $\bb$}
\author{Micha{\l} \'Swi\c{e}tek}

\begin{document}

\maketitle
\begin{abstract}
    The space $\bb[(m_j)_j,(n_j)_j]$ is a Bourgain-Delbaen space modelled on a mixed Tsirelson space $\mT[(m_j)_j,(n_j)_j]$
    and is a slight modification of $\B[(m_j)_j,(n_j)_j]$, a space defined by S. Argyros and R. Haydon.
    We prove that in every infinite dimensional subspace of $\bb[(m_j)_j,(n_j)_j]$ there exists a basic sequence equivalent to
    a sequence of weighted basis averages of increasing length from $\mT[(m_j)_j,(n_j)_j]$. 
    We remark that the same is true for the original space $\B[(m_j)_j,(n_j)_j]$.
\end{abstract}

\section{Introduction}

In 1980, Bourgain and Delbaen \cite{BD} discovered a new general scheme of constructing separable $\Li$-spaces, i.e.
spaces of the form $\overline{\bigcup_{n\in\N} F_n}$, where $F_n \subset F_{n+1}$ and $F_n$ is $C$-isomorphic to $\ell_\infty^{k_n}$,
for every $n\in\N$, some uniform constant $C$, and some sequence $(k_n)_n\subset\N$.
Using the scheme they constructed two classes of new isomorphic preduals to $\ell_1$,
which answered many open problems.
Among others, they positively solved the problem of existence of a predual of $\ell_1$ not containing $c_0$.
The novelty of their method relied on usage of isomorphic copies of finite dimensional $\ell_\infty^n$ spaces instead of isometric ones.
A space constructed by this scheme is nowadays called a Bourgain-Delbaen space.

During last 10 years, the Bourgain-Delbaen scheme attracted attention of many researchers who used it to construct new spaces or prove general theorems.
To mention only a few:
in \cite{AH} the authors construct a hereditarily indecomposable Banach space $\ah$ with the scalar-plus-compact property,
in \cite{FOS} the authors prove the universality of $\ell_1$ as a dual Banach space, and
in \cite{AFHO} the authors prove that every separable uniformly convex space can be embedded into a Banach space with the scalar-plus-compact property.
Remarkably, the scheme turned out to be the most general way of constructing separable $\Li$-spaces, as
every separable $\Li$-space is isomorphic to a Bourgain-Delbaen space.
It was proved in \cite{AGM}.

The ambient space for the construction of $\ah$ was the space $\B=\B[(m_j)_j,(n_j)_j]$, for some fixed sequences $(m_j)_j,(n_j)_j$ of natural numbers.
It is a Bourgain-Delbaen space, which is modelled on a mixed Tsirelson space $\mT[(m_j)_j,(n_j)_j]$.
Consequently, it is unconditionally saturated and does not contain $c_0$ nor $\ell_p,\ p\in[1,\infty)$.
Lately, the space $\B$ was used in \cite{AM} and in \cite{MPS}.
In \cite{AM} the authors construct an example of a hereditarily indecomposable $\Li$-space $\am$ such that
it does not contain $c_0$, $\ell_1$, or reflexive subspaces and has the scalar-plus-compact property.
In \cite{MPS} the authors construct an example of an $\Li$-space $\mps$ with the scalar-plus-compact property, but
with the opposite structure inside the space to the structure of the spaces $\ah$ and $\am$, namely, the space $\mps$ is unconditionally saturated, whereas the spaces $\ah$ and $\am$ are hereditarily indecomposable.
Both spaces $\am$ and $\mps$ are quotients of $\B$, constructed using the self-determined-sets technique introduced in \cite{AM}.

R. Haydon proved \cite{H} that one of the first examples constructed by Bourgain and Delbaen within their scheme,
the space $X_{a,b}$ for $0<b<1/2<a<1$ and  $a+b>1$, is saturated with $\ell_p$, where $p\in(1,\infty)$ is determined by equations $1/p+1/q=1$ and $a^q+b^q=1$.
In \cite{GPZ} the authors introduce spaces $\gpz$ for $p\in(1,\infty)$ modelled on the Tsirelson space $T(\A_n,\overline b)$,
for some fixed $n>1$, and a sequence $\overline b=(b_1,\dots,b_n)$ of positive real numbers satisfying $b_1<1,\ b_2,\dots,b_n<1/2$ and
$\sum_{i=1}^n b^q=1$.
The space $T(\A_n,\overline b)$ is isomorphic to $\ell_p$ and the space $\gpz$ is saturated with $\ell_p$.
Moreover, the authors notice that for $n=2$ their definition of the space $\gpz$ essentially coincide with $X_{b_1,b_2}$.

Continuing this line of research we prove the following theorem (see Section 4 for more precise formulation)

\begin{theorem:a}
The space $\bb[(m_j)_j,(n_j)_j]$ is saturated with sequences of weighted basis averages of increasing length from the mixed Tsirelson space $\mT[(m_j)_j,(n_j)_j]$.
\end{theorem:a}

The space $\bb[(m_j)_j,(n_j)_j]$ is a modification of $\B[(m_j)_j,(n_j)_j]$.
The modification is technical, and was made in order to slightly minimise the notational complexity.
Nevertheless, we make comments in the paper showing that Theorem \ref{thm:a} is true for $\B[(m_j)_j,(n_j)_j]$ as well.

The paper is organized as follows.
In Section 2 we recall basic facts and definitions, i.e. the mixed Tsirelson spaces, the space $\B$, different types of analyses of nodes,
and rapidly increasing sequences.
Section 3 contains lemmas used in the proof of Main Theorem. 
Lemmas \ref{split} and \ref{comp} may be of independent interest.
Finally, in Section 4 we give a proof of Theorem \ref{thm:a}.

This work is a part of author's Ph.D. thesis, which was written under the direction of
Anna Pelczar-Barwacz. 
The author would like to thank her for her kind introduction to the
field, valuable discussions and insightful remarks.

\section{Basic Facts and Definitions.}

\subsection{The mixed Tsirelson spaces.} 
They originated in \cite{AD} and generalise the original Tsirelson
space \cite{T} and the Schlumprecht space \cite{S} (see Remark \ref{mtex}).
\begin{definition}
Let $(\M_n)_{n\in\N}$ be a sequence of compact subsets of $\mathcal P(\N)$, $(\theta_n)_{n\in\N}$ be a sequence of positive reals, and 
$W=W[(\M_n,\theta_n)_n]$ be the smallest subset of $c_{00}$ satisfying:
\begin{enumerate}
    \item $\pm e^*_i \in W$ for all $i\in \N$,
    \item for all $n,d\in \N$, $f_1<\cdots<f_d\in W[(\M_n,\theta_n)_n]$
        if $(\min\supp f_i)_{i=1}^d\in \M_n$, then the combination $\theta_n\sum_{i\leq d}f_i$ is in $W[(\M_n,\theta_n)_n]$.
\end{enumerate}
Define $\mT[(\M_n,\theta_n)_n]$ to be the completion of $(c_{00}, \|\cdot\|_{W[(\M_n,\theta_n)_n]})$, where
$\|x\|_{W[(\M_n,\theta_n)_n]} = \sup \{ f(x) \mid f\in W[(\M_n,\theta_n)_n]\}$,
for all $x\in c_{00}$.
\end{definition}

\begin{remark}\label{mtex}
Define $\A_n = \{F\subset \N\mid \# F = n\}$, for all $n\in\N$.
\begin{enumerate}
\item If $\M_n=\M$, $\theta_n=\theta$ for every $n\in\N$ and some $\M, \theta$, then
    \begin{itemize}
        \item if $i(\M)<\omega$ and $\frac{1}{i(\M)}\geq \theta$, then $\mT[(\M_n,\theta_n)_n]\cong c_0$,
        \item if $i(\M)<\omega$ and $\frac{1}{i(\M)}< \theta$, then $\mT[(\M_n,\theta_n)_n]\cong \ell_p$ for some $p\in (1,\infty)$,
    \end{itemize}
    \item If $\M_n = \mathcal S=\{\,F\subset \N\mid \#F\leq \min F\,\}$ and $\theta_n = \frac{1}{2}$, for every $n\in\N$, then $\mT[(\M_n,\theta_n)_n]$ is the Tsirelson space,
    \item if $\M_n = \A_n$ and $\theta_n = \frac{1}{\log(n+1)}$, for every $n\in\N$, then $\mT[(\M_n,\theta_n)_n]$ is the Schlumprecht space. 
\end{enumerate}
\end{remark}

The following theorem describes basic properties of the mixed Tsirelson spaces needed in the sequel.
\begin{theorem}[Theorem I.10 \cite{AT}]\label{basic-mt}
The standard vectors $(e_n)_{n\in\N}\subset c_{00}$ are $1$-unconditional basis for $\mT[(\M_n,\theta_n)_n]$ and
if for some $n\in\N$ it holds that the Cantor–Bendixson index $i(\M_n)\geq\omega$ or $i(\M_n)=r<\omega$ and $\theta_n > \frac{1}{r}$, then the space $\mT[(\M_n,\theta_n)_n]$ is reflexive.

Moreover, if the first alternative holds or the second alternative holds for some increasing sequence $(r_{n_k})_{k\in\N}\subset \N$, 
then the space $\mT[(\M_n,\theta_n)_n]$ does not contain any of the spaces $c_0,\ \ell_p$, for $p\in [1,\infty)$.
\end{theorem}

A tree analysis of a functional from $W$ is the main tool for bounding the norm of a vector in a mixed Tsirelson space.
\begin{definition}\label{tree-analysis-mt}
Let $f\in W[(\M_n,\theta_n)_n]$, let $\T$ be a tree with the root $\emptyset$, and let $S_t$ denote the set of all successors of $t$ in $\T$, for all $t\in \T$. 
We call a sequence $(f_t)_{t\in\T}$ a tree analysis of $f$ if
\begin{itemize}
    \item $f_\emptyset=f$,
    \item if $t\in \T$ is a terminal node, then $f=\pm e^*_k$ for some $k\in\N$,
    \item if $t\in \T$ is a non-terminal node, then $f_t=\theta_n\sum_{s\in S_t} f_s$ for some $n\in\N$ and some block sequence  $(f_s)_{s\in S_t} \subset W[(\M_n,\theta_n)_n]$ with $(\min\supp f_s)_{s\in S_t} \in \M_n$.
\end{itemize}
\end{definition}

As a consequence of the minimality of $W$ we obtain
\begin{proposition}
    Every $f\in W[(\M_n,\theta_n)_n]$ admits a tree analysis.
\end{proposition}

For the rest of this paper we fix sequences $(m_j)_{j \in \N}$, $(n_j)_{j \in \N}$ of natural numbers.
We need certain growth conditions for them.
\begin{assumption}
We assume that
\begin{enumerate} 
    \item $m_1 = n_1 = 4$,
    \item $m_{j+1} \geq m_j^2$,
    \item $n_{j+1} \geq m_{j+1}^2(4n_j)^{\log_2 m_{j+1}}$.
\end{enumerate}
\end{assumption}

\begin{notation}
    We will write $\mT=\mT[(m_j)_j,(n_j)_j]=\mT[(\A_{n_j},m_j^{-1})_j]$.
\end{notation}
By Theorem \ref{basic-mt} the space $\mT$ is a reflexive space and the standard vectors $(e_n)_{n\in\N}\subset c_{00}$ are an unconditional basis for it.

\begin{proposition}[Lemma II.9 \cite{AT}] \label{baver-norm}
Let $(z_i)_{i=1}^{n_j}$ be a subsequence of the standard basis $(e_n)_{n\in\N}$ in $\mT$.
Then 
\[
    \|\frac{m_j}{n_j}\sum_{i=1}^{n_j} z_i\| = 1.
\]
\end{proposition}

\subsection{The space $\bb$}

We define the space $\bb$ below.
It is a slight modification of the space $\B$.
The construction is a special case of a general scheme for constructing separable $\Li$-spaces, which was defined for the first time in \cite{BD}.
We follow the slightly modified version of the construction from \cite{AH}, where the space $\B$ itself was defined.

Fix an increasing sequence $(N_q)_{q\in\N}$ of natural numbers.
We define inductively a sequence of disjoint finite sets $(\Delta_q)_{q\in\N}$.
Let $\Delta_1=\{1\}$. 
Assume that sets $\Delta_1,\dots,\Delta_q$ have been defined.
Set $\Gamma_0 = \emptyset$, $\Gamma_p = \bigcup_{r=1}^p \Delta_r$, $p\leq q$, and

\begin{align*}
    \Delta_{q+1}  & =  \bigcup_{j=1}^{q+1} \{ (q+1, 0, m_j,  \varepsilon e^*_\eta) \mid  \varepsilon=\pm 1,\ \eta \in \Gamma_q \} \cup
 \bigcup_{1\leq p<q} \bigcup_{j=1}^{p} \{ (q+1, \xi, m_j, \varepsilon e^*_\eta) \mid  \\
& \quad\quad\quad\quad\quad\quad\xi\in \Delta_p,\   
\w(\xi)=m_j^{-1}, \age(\xi)<n_j, \varepsilon=\pm 1,\ \eta \in \Gamma_q\setminus\Gamma_p \},
\end{align*}
where $\w(q, \xi, m_j,  \varepsilon e^*_\eta) = m_j^{-1}$, $\age(q, 0, m_j,  \varepsilon e^*_\eta)=1$,
and $\age(q, \xi, m_j,  \varepsilon e^*_\eta) = \age(\xi)+1$.
For a node $\gamma=(q+1, \xi, m_j, \varepsilon e^*_\eta)$ we define $\rank(\gamma) = q+1$.

Let $\Gamma = \bigcup_{q\in\N}\Delta_q$, and for every 
$\gamma\in\Gamma$ we define
\[
    c^*_\gamma =
    \begin{cases}
        m_j^{-1}P^*_{(p,q]} \varepsilon e^*_\eta, &\text{ for } 
        \gamma=(q+1, 0, m_j, \varepsilon e^*_\eta), \\
        e^*_\xi+m_j^{-1} P^*_{(p,q]}\varepsilon e^*_\eta, &\text{ for } 
        \gamma=(q+1, \xi, m_j,  \varepsilon e^*_\eta),
    \end{cases}
\]
and $d^*_\gamma = e^*_\gamma - c^*_\gamma$, where $P^*_{(p,q]}$ is a projection onto $\langle d^*_\gamma \mid \rank(\gamma)\in(p,q]\rangle$.

\begin{remark}
In \cite{AH} authors use projections of the form $P^*_{(p,\infty)}$.
Notice that
\[
    P^*_{(p,\infty)} \restriction \ell_1(\Gamma_q) = P^*_{(p,q]} \restriction \ell_1(\Gamma_q).
\]
\end{remark}

By Theorem 3.5 \cite{AH} the sequence $(d^*_\gamma)_{\gamma\in\Gamma}$ is
a basis for $\ell_1(\Gamma)$ and we take $(d_\gamma)_{\gamma\in\Gamma}\subset \ell_\infty(\Gamma)$ to be a biorthogonal sequence to 
$(d^*_\gamma)_{\gamma \in \Gamma}$.

\begin{definition}
We define
\[
    \bb = 
    \overline{\langle d_\gamma \mid \gamma \in \Gamma \rangle}
    \subset \ell_\infty(\Gamma).
\]
\end{definition}

\begin{remark}
The difference between our definition of the space $\bb$ and the definitions of the space $\B$ from \cite{AH} and \cite{AM} is in the form of nodes.
Indeed, we allow nodes of the form $(q+1, \xi, m_j^{-1}, \varepsilon e^*_\eta)$, whereas they allow $(q+1, \xi, m_j^{-1}, b^*)$, where $b^*$ is 
from a finite net in the unit ball of $\ell_1(\Gamma_q\setminus\Gamma_p)$.
We decided to present the simpler version to not further complicate already quite technical proofs, but our results are 
still true for $\B$ - the version from \cite{AH} and \cite{AM}.
To obtain a full proof for $\B$ we should consider convex combinations of functionals $\pm e^*_\eta$ in the definition of tree-analysis, what would add an additional level of notational complexity in the following proofs.
On the other hand, it is easily seen that convex combinations do not break correctness of the proofs.
\end{remark}

\begin{notation}
    For a functional $h\in \langle d^*_\gamma \mid \gamma \in \Gamma \rangle$ we define $\rng h = [p,q]$, 
    where $p,q$ are respectively maximal, minimal with $h \in \langle d^*_\gamma \mid \rank(\gamma) \in [p,q] \rangle$.
    Fox a block $ x \in \langle d_\gamma \mid \gamma \in \Gamma \rangle$ we define $\rng x = [p,q]$,
    where $p,q$ are respectively maximal, minimal with $x \in \langle d_\gamma \mid \rank(\gamma) \in [p,q] \rangle$.
\end{notation}

The space $\bb$ is an $\Li$-space with dual isomorphic to $\ell_1$ and saturated by reflexive subspaces with an unconditional basis \cite{AH}.

We introduce different types of analysis of evaluation functionals following  \cite{AH} and \cite{GPZ}, adjusting their scheme to our situation. 

\subsection{Different types of analyses of evaluation functionals} 

\ \\
\textbf{The evaluation analysis of $e_{\gamma}^*$}.

First, we notice that every $\gamma\in\Gamma$ admits a unique analysis as follows (Prop. 4.6 \cite{AH}).
Let $\w(\gamma)=m_{j}^{-1}$.
Then using backward induction we determine a sequence of sets $(I_i,\varepsilon_i e_{\eta_i}^*,\xi_i)_{i=1}^{a}$ 
so that $\xi_a=\gamma$, $\xi_1=(\max I_1+1,0,m_j, \varepsilon_1 e_{\eta_1}^*),$ 
$\xi_i=(\max I_i,\xi_{i-1},m_j, \varepsilon_i  e^*_{\eta_i}),$ 
and $\max I_{i-1}+2 = \min I_i$
for every $1<i\leq a$.

Repeating the reasoning of \cite{AH}, as $e^*_{\xi}=d^*_{\xi}+c^*_{\xi}$ for each $\xi\in\Gamma$, with the above notation we have
\[
e^*_{\gamma}= \sum_{i=1}^{a}d^*_{\xi_{i}}+m_{j}^{-1}\sum_{i=1}^{a}\varepsilon_{i} P^*_{I_i}e^*_{\eta_{i}}.
\]

\begin{definition}
Let $\gamma\in\Gamma$.
Then the sequence $(I_i,\varepsilon_i e_{\eta_i}^*, \xi_i)_{i=1}^{a}$ satisfying all the above properties will be called the evaluation analysis of $\gamma$.

We define the bd-part and mt-part of $e^*_\gamma$ as 
\[
\bdp(e^*_\gamma)=\sum_{i=1}^{a} d_{\xi_i}^*,  \   
\mt(e^*_\gamma)=m_{j}^{-1}\sum_{i=1}^{a}\varepsilon_{i} P^*_{I_i}e^*_{\eta_{i}}.
\]
\end{definition}

\begin{remark}\label{mt-part}
Fix $(\eta_s)_{s=1}^a\subset\Gamma$, $(I_s)_{s=1}^a$ a sequence of intervals of natural numbers with $\max I_{s-1}+2 = \min I_s$,
and $(\varepsilon_s)_{s=1}^a\subset \{\pm 1\}$,  with $a\leq n_j$, $j\leq q_1$,
$\eta_s\in\Gamma_{\max I_s}\setminus \Gamma_{\min I_s-1}$, $s=1,\dots,a$.

Then the formulae $\xi_1=(\max I_1+1,0,m_j, \varepsilon_1  e^{*}_{\eta_1})$ and 
$\xi_s=(\max I_s+1,\xi_{s-1},m_j, \varepsilon_s e^*_{\eta_s})$, for any $s\leq a$, give well-defined nodes.

\end{remark}

\textbf{The $I$(interval)-analysis of a functional $e_{\gamma}^*$}.

Fix $I\subset\N$ and $\gamma\in\Gamma$. 
Let $\w(\gamma)=m_{j}^{-1}$, $a\leq n_{j}$ and $(I_i,\varepsilon_i e_{\eta_i}^*, \xi_i)_{i=1}^{a}$ be the evaluation analysis of $\gamma$.
We define the $I$-analysis of $e_{\gamma}^*$ as follows:
\begin{enumerate}

\item[(a)] If for at least one $i$ we have $P^*_{I_i\cap I}e^*_{\eta_i}\neq 0$, then the $I$-analysis of $e_{\gamma}^*$ is of the following form
\[
(I_i\cap I,\varepsilon_i e^*_{\eta_i},\xi_i)_{i\in A_I}, 
\]
where $A_I=\{i \mid \ P^*_{I_i\cap I}e^*_{\eta_i}\neq 0\}$. In this case we say that $e_{\gamma}^*$ is $I$-decomposable.
\item[(b)] If $P^*_{I_i\cap I}e^*_{\eta_i}=0$ for all $i=1,\dots,a$, then we assign no $I$-analysis to $e_{\gamma}^*$ and we say that $e_{\gamma}^*$ 
is $I$-indecomposable. 
\end{enumerate}

Now we introduce the tree-analysis of $e_{\gamma}^*$ analogous to the tree-analysis of a functional in a mixed Tsirelson space (see \cite{AT} Chapter II.1).

We start with some notation. We denote by $(\T,\preceq)$ a finite tree, whose elements are finite sequences of natural numbers ordered by the initial segment partial order. Given $t\in\T$ denote by $S_t$ the set of immediate successors of $t$.

Let $(I_t)_{t\in\T}$ be a tree of intervals of $\N$ such that $t\preceq s$ iff $I_t\supset I_s$ and $t,s$ are incomparable iff $I_t\cap I_s=\emptyset$. 

\textbf{The tree-analysis of a functional $e_{\gamma}^*$}.

Let $\gamma\in\Gamma$. The tree-analysis of $e_{\gamma}^*$ is a family of the form $(I_t,\varepsilon_t,\eta_t )_{t\in\T}$ defined inductively in the following way:
\begin{enumerate}
\item $\T$ is a finite tree with a unique root denoted by $\emptyset$.
\item Set $\eta_{\emptyset}=\gamma$, $I_{\emptyset}=(0,\rank\gamma)$, $\varepsilon_\emptyset=1$ and let $(I_i,\varepsilon_i e_{\eta_i}^*, \xi_i)_{i=1}^{a}$ be the evaluation analysis of $e^*_{\eta_{\emptyset}}$. 
Set $S_{\emptyset}=\{(1),(2),\ldots,(a)\}$ and for every $s=(i)\in S_{\emptyset}$ define $(I_s,\varepsilon_s,\eta_s)=(I_i,\varepsilon_i,\eta_i)$.
\item Assume that for $t\in\T$ the tuple $(I_t,\varepsilon_t,\eta_t)$ is defined. Let $(I_i,\varepsilon_i e^*_{\eta_i}, \xi_i)_i$ be the evaluation analysis of $e_{\eta_t}^*$. Consider two cases:
\begin{enumerate}
\item If $e_{\eta_t}^*$ is $I_t$-decomposable, let $(I_i,\varepsilon_i e^*_{\eta_i}, \xi_i)_{i\in A_{I_t}}$ be the $I_t$-analysis of $e_{\eta_t}^*$. Set $S_t=\{(t^\smallfrown i): i\in A_{I_t}\}$. For every $s=(t^\smallfrown i)\in S_t$, let $(I_s,\varepsilon_s,\eta_s)=(I_i,\varepsilon_i,\eta_i)$.
 \item If $e_{\eta_t}^*$ is $I_t$-indecomposable, then $t$ is a terminal node of the tree-analysis. 
\end{enumerate}
\end{enumerate}

\begin{remark}
For every $\gamma\in\Gamma$ its tree analysis $(I_t,\varepsilon_t,\eta_t )_{t\in\T}$ is uniquely defined.
This is in contrast to the mixed-Tsirelson case.
Moreover, for every $t_0\in\T$ the tree analysis of $e^*_{\eta_{t_0}}$ restricted to $I_{t_0}$ is $(I_t,\varepsilon_t,\eta_t )_{t\in\T,\, t\succeq t_0}$.
\end{remark}

\begin{notation}
    For every $\gamma \in\Gamma$, its evaluation analysis $(I_i,\varepsilon_i e_{\eta_i}^*, \xi_i)_{i=1}^{a}$, its tree analysis $(I_t, \varepsilon_t, \eta_t)_{t\in\T}$, 
    and a subset $\T'\subseteq \T$ we write 
    \begin{enumerate}
        \item $m_t = m_{i_0}$, $n_t=n_{i_0}$, where $i_0\in\N$ is such that $m_{i_0}=\w(\eta_t)^{-1}$, for all $t\in \T$,
        \item $|\bdp|(e^*_\gamma)=\sum_{i=1}^{a} |d_{\xi_i}^*|$,
        \item $|\bdp|(e^*_\gamma, \T') = 
                \sum_{t \in \T\cap\T'} |\bdp|(e^*_{\eta_t})\circ P_{I_t}$.
    \end{enumerate}
\end{notation}

The following notion is very helpful in regularising ranges of nodes in a tree analysis of a given node
versus ranges of a given block sequence.
\begin{definition}
Let $(x_k)_k \in \bb$ be a block sequence, $\gamma\in\Gamma$,
and $(I_t, \varepsilon_t, \eta_t)_{t\in\T}$ be the tree analysis of $e^*_\gamma$.
\begin{enumerate}
    \item We say that $e^*_{\eta_t}P_{I_t}$ covers $x_k$ if 
    $t\in\T$ is maximal in $\T$ with
    $\rng(x_k)\cap\rng(e^*_\gamma) \subset 
    \rng(e^*_{\eta_t}P_{I_t})$.
    We call $t$ the covering index for $x_k$.
    \item We say that a finite sequence $(A_l)_{l=1}^i$ of finite intervals $\N$ is comparable with the sequence $(x_k)_k$ if for all $l,k$
    \[
        A_l\subseteq \rng x_k \text{ or }
        \supp x_k \cap \bigcup_{l=1}^iA_l \subseteq A_k \text{ or }
        A_l \cap \rng x_k = \emptyset.
    \]
    \item We say that $e^*_\gamma$ is comparable with $(x_k)_k$ if
    the sequence $(\rng e^*_{\eta_t}P_{I_t})_{t\in\T}$ is
    comparable with $(x_k)_k$.
\end{enumerate}
\end{definition}

\subsection{Rapidly Increasing Sequences}
Vectors of the following form are the building blocks for the target sequences of interest, i.e. RISes and two-level RISes.
\begin{definition}
Let $X$ be a Banach space with a basis, $x \in X$, and $C \geq 1$.
We call $x$ a $C-\ell^i_1$ $average$ if there is a block sequence 
$(x_l)_{l=1}^i$ of vectors from $X$ with norms uniformly bounded from above by $C$ and such that
$x=\frac{1}{i}\sum_{l=1}^i x_l$.
\end{definition}

\begin{lemma}[Lemma 8.2, 8.3 \cite{AH}]\label{average-bound}
Let $C\geq 1$, $i\in \N$.
Then
\begin{enumerate}
    \item there exists a $C-\ell^i_1$ average in every infinite dimensional 
        subspace of $\bb$,
    \item for every $\gamma \in \Gamma$ and $C-\ell^i_1$ average  $x\in\bb$ we have 
        $|d^*_\gamma(x)| \leq \frac{3C}{i}$.
\end{enumerate}
\end{lemma}

\begin{lemma}[Lemma II.23 \cite{AT}]\label{norm-proj}
Let $C\geq 1$, $j\leq i$, let $(E_l)_{l=1}^j$ be a sequence of pairwise disjoint intervals and let $x$ be a $C-\ell_1^i$-average in some Banach space.
Then
\[
    \sum_{l=1}^j \|P_{E_l}x\| \leq (1+\frac{2j}{i})\sup_{1\leq l\leq j}\|P_{E_l}x\| .
\]
\end{lemma}

\begin{remark}\label{aveproj}
It is easy to see that the above Lemmas \ref{average-bound} and \ref{norm-proj} are true for projections of averages on intervals of positive integers.
\end{remark}

Now we define the main objects of interest.
The crucial property they enjoy is that they behave like the basis of $\bb$ but can be found in every subspace.
\begin{definition}\label{RIS}
Let $I\subseteq\N$ be an interval and $(x_i)_{i\in I}$ be a block sequence in $\bb$.
We call it a \textit{rapidly increasing sequence} (RIS) 
if there exist a constant $C\geq 1$ and an increasing sequence 
$(j_i)_{i\in I}$ (growth index) such that for all $i$ we have
\begin{enumerate}
    \item $\|x_i\| \leq C$,
    \item $j_{i+1} > \max \rng x_i$,
    \item $|x_i(\gamma)| \leq Cm_\gamma^{-1}$ for all 
        $\gamma \in \Gamma$ with $m_\gamma < m_{j_i}$.
\end{enumerate} 
\end{definition}

In what comes later we need additional structure in RIS, namely
\begin{definition}
Let $I\subseteq\N$ be an interval.
We call a sequence $(y_i)_{i\in I}\subset \bb$ \textit{a two-level $C$-RIS} if $(y_i)_{i\in I}$ is a normalised $C$-RIS such that
for each $i\in I$ there are a sequence  
$n_{j_i}<n_{j_{i,1}}<\cdots<n_{j_{i,n_{j_i}}}$ of positive integers and a  normalised $C$-RIS $(y_{i,j})_{j=1}^{n_{j_i}}$ such that
$y_i = \frac{c_im_{j_i}}{n_{j_i}} \sum_{j=1}^{n_{j_i}}y_{i,j}$, where
$y_{i,j}$ is a normalised $C-\ell_1^{n_{j_{i,j}}}$ average with $\max \rng y_{i,j} + j_{i,j} <\min\rng y_{i,j+1}$,
for $j=1,\dots,n_{j_i}$, and $c_i>0$ is some normalising constant.
\end{definition}

\begin{proposition}[Prop. 5.6 \cite{AH}]\label{basic-ineq}
Let $C\geq 1$ and let $(x_i)_{i=1}^{n_j}$ be a $C$-RIS in $\bb$. 
Then 
\[
    \| \frac{m_j}{n_j} \sum_{i=1}^{n_j} x_i \|
    \leq 10C.
\]
Moreover, for all $\gamma \in \Gamma$ with $n_\gamma > n_j$
\[
    |\frac{m_j}{n_j}\sum_{i=1}^{n_j} x_i(\gamma)|
    \leq \frac{10C}{m_j}.
\]
\end{proposition}

\begin{corollary}\label{cis-bounds}
For $c_i$'s in the definition of two-level RIS we have 
$\frac{1}{10C} \leq c_i \leq 1$.
\end{corollary}

\begin{lemma}[Corollary 8.5\cite{AH}] \label{er}
Let $C>2$. 
In every infinite dimensional subspace of $\bb$ there exists a $C$-RIS $(x_i)_i$ of arbitrary length, 
arbitrary growth index $(j_i)_i$, and such that for every $i$ the vector $x_i$ is an $\ell^{n_{j_i}}_1$-average.
\end{lemma}

Lemma \ref{er} and Proposition \ref{basic-ineq} imply the following
\begin{corollary}
Let $C>2$. 
In every infinite dimensional subspace of $\bb$ there exists a two-level $C$-RIS of arbitrary length.
\end{corollary}

\section{Main Lemmas.}

\begin{definition}
Let $\T$ be a tree, $u\preceq v \in \T'$. 
We say that a sequence $(u_l)_{l=0}^r$ is a branch of length $r$ from $u$ to $v$, if
$u_0 = u$, $u_{l+1}\in S_{u_l},\ l=1,\dots,r-1$, and $u_r=v$.
\end{definition}

\begin{lemma}\label{lenpath}
Let $i\in \mathbb N$, $\T$ be a tree with the root equal $\emptyset$,
$v\succ\emptyset$, 
$b=(u_l)_{l=0}^r$ be a branch from $\emptyset$ to $v$.
\begin{enumerate}
   \item Let $\gamma \in \Gamma$,  $(I_t, \varepsilon_t, \eta_t)_{t\in\T}$ be its tree analysis,
         $y\in \bb$ be a normalised vector and assume that
         \begin{enumerate}
             \item $v$ is the covering index for $y$,
             \item for all $\xi\in\Gamma$ we have $|d^*_\xi(y)|<4^{-i-3}$,
             \item  $|e^*_\gamma(y)| > 4^{-i-1}$.
         \end{enumerate}
    Then $r \leq i+1$.
    \item Let $f\in W$, $(f_t)_{t\in\T}$ be its tree analysis, $z=e_l\in \mT$, for some $l\in\N$, and 
        assume that
        \begin{enumerate}
             \item $f_v=z^*_l$,
             \item $f(z) > 4^{-i-1}$.
         \end{enumerate}
        Then $r \leq i$.
\end{enumerate}
\end{lemma}
\begin{proof}
$(1)$.
Using the evaluation analysis and the assumption $u_r = v$ we have
\begin{align*}
    e^*_\gamma(y) 
        & = \sum_{s\in S_{u_0}} d^*_{\beta_s}(y) +
            \frac{1}{m_{u_0}}\sum_{s\in S_{u_0}}\varepsilon_s e^*_{\eta_s}P_{I_s}(y) \\
        & = \sum_{s\in S_{u_0}} d^*_{\beta_s}(y) +
            \frac{1}{m_{u_0}}\varepsilon_{\eta_{u_1}} e^*_{\eta_{u_1}}P_{I_{u_1}}(y) \\
        & = \sum_{s\in S_{u_0}} d^*_{\beta_s}(y) +
            \frac{\varepsilon_{\eta_{u_1}}}{m_{u_0}} \left( \sum_{s\in S_{u_1}} d^*_{\beta_s}(y) +
            \frac{1}{m_{u_1}}\sum_{s\in S_{u_1}}\varepsilon_{\eta_s}e^*_{\eta_s}P_{I_s}(y)) \right) \\
        & =  \sum_{l=0}^{r-1} \left( \prod_{a=0}^{l-1}\frac{\varepsilon_{u_{a+1}}}{m_{u_a}} \sum_{s\in S_{u_l}} d^*_{\beta_s}(y) \right) +
            \left( \prod_{a=0}^{r-1} \frac{\varepsilon_{u_{a+1}}}{m_{u_a}} \right)
            e^*_{\eta_{u_r}}P_{I_{u_r}}(y). 
\end{align*}
Moreover, the assumption $u_r = v$ implies also that for at most one $l\in\{\,0,\dots,r-1\,\}$ there are two $s \in S_{u_l}$
such that $d^*_{\beta_s}(y) \neq 0$ and for the rest $l$'s from  $\{\,0,\dots,r-1\,\}$ there are at most one such $s$. 
Hence by $(c)$ we have
\[
     4^{-i-1} < |e^*_\gamma(y)| \leq 2\cdot 4^{-i-3} \sum_{l=0}^{r-1} \left(
    \prod_{a=0}^{l-1}\frac{1}{m_{u_a}}  \right) +
            3\left( \prod_{a=0}^{r-1} \frac{1}{m_{u_a}} \right)
            \leq 4^{-i-2} + 3\cdot 4^{-r}.
\]
It follows that
\[
   3\cdot 4^{-r} \geq 4^{-i-1} - 4^{-i-2} > 4^{-i-2},
\]
which finishes the proof of $(1)$.

$(2)$.
The second part can be proved in a similar manner as (1). 
Indeed, it is enough to repeat the arguments omitting the Bourgain-Delbaen parts of functionals.
\end{proof}

In translating $\bb$ to $\mT$ we need to deal with bd-parts.
The following lemma solves one of the frequently occurring situations.
\begin{lemma}\label{est-bd}
Let $(y_i)$ be a two-level $C$-RIS, $(a_i)$ a sequence of scalars of modulus bounded by $1$, 
$\gamma\in\Gamma$ with the tree analysis $(I_t, \varepsilon_t, \eta_t)_{t\in\T}$ 
satisfying $|e^*_\gamma(y_i)| > 4^{-i-1}$. 
For every $i$ let $t_i$ be the covering index for $y_i$ and let $\T'$ be the minimal subtree of $\T$ containing the root and all $t_i'$s.
If $n_{t_i}\leq n_{j_i}$ for every $i$, then
\[
    |\bdp|(e^*_\gamma, \T')(\sum_i a_iy_i) \leq 1/n_{j_1}.
\]
\end{lemma}
\begin{proof}
Fix $i$ and let $(u_l)_{l=0}^r\subset \T$ be a branch from $\emptyset$ to $t_i$.
The definition of a two-level RIS and Lemma \ref{average-bound} gives for every $i$ and every $\xi\in\Gamma$
that 
\[
    |d^*_\xi(y_i)| \leq \frac{m_{j_i}}{n_{j_i}}\frac{3C}{n_{j_{i,1}}}.
\]
By Lemma \ref{lenpath} we know that $r\leq i+1$, so we have
\begin{align*}
    |\bdp(e^*_\gamma,\T')|(y_i) 
    & = \sum_{l=0}^{r-1}  \sum_{s\in S_{u_l}} |d^*_{\beta_s}(y_i)| +
     \sum_{s\in S_{t_i}}|d^*_{\beta_s}(y_i)| \\
     & \leq \sum_{l=0}^i 2\frac{m_{j_i}}{n_{j_i}}\frac{3C}{n_{j_{i,1}}} +
     n_{j_i}\frac{m_{j_i}}{n_{j_i}}\frac{3C}{n_{j_{i,1}}} 
      \leq \frac{3C}{n_{j_{i,1}}} \left(2(i+1)\frac{m_{j_i}}{n_{j_i}} + m_{j_i}\right) \\
      & \leq \frac{1}{2n_{j_i}}.
\end{align*}
Summing for all $i$'s and using $|a_i|\leq 1$ we obtain the statement.
\end{proof}

\begin{lemma}[Lemma 5.2 \cite{AH}] \label{gonRIS}
Let $I\subseteq\N$ be an interval, $(x_i)_{i\in I}$ be a $C$-RIS, $(j_i)_{i\in I}$ be its growth index.
For every $\gamma \in \Gamma $ and $s \in \N$ we have
\[
    |e^*_\gamma P_{(s,\infty)}(x_i)| \leq
    \begin{cases}
        5Cm_\gamma^{-1} & \text{, if }  m_\gamma < m_{j_i}, \\
        3Cm_\gamma^{-1} & \text{, if }  m_\gamma \geq m_{j_{i+1}}
    \end{cases}
\]
\end{lemma}

The following Proposition describes a possibility of approximation of a given node by a node with controlled rank.
\begin{proposition}\label{split}
Let $\gamma \in \Gamma$, $K,N\in \N$, $K\leq\min\rng e^*_\gamma \leq N$.
For every $i\in \N$ there exists $\gamma'\in \Gamma_{N+i}$ such that $K\leq\min\rng e^*_{\gamma'}$ and
\[
    \|(e^*_\gamma-e^*_{\gamma'})\restriction \langle d_\xi \mid \xi \in \Gamma_N \rangle\| \leq 4^{-i}.
\]
\end{proposition}
\begin{proof}
If $\gamma\in \Gamma_{N+i}$, then $\gamma'=\gamma$ satisfies the conditions of the Lemma.
Assume that $\gamma\in \Gamma\setminus\Gamma_{N+i}$.
Consider the tree analysis  $(I_t, \varepsilon_t, \eta_t)_{t\in\T}$ of $e^*_\gamma$
and for every $t\in\T$ let $\bdp(e^*_{\eta_t})P_{I_t}=\sum_{s\in S_t} d^*_{\beta_s}$.

Define $u_0 = \emptyset$ and 
$u_1 = \max \{s\in S_{u_0} \mid \min\rng (e^*_{\eta_s}P_{I_s}) \leq N\}$.
If $\rank (d^*_{\beta_{u_1}})  \leq N + i$, then we define 
$\eta_{u_0}' = \beta_{u_1}$ and we are done.

Assume that $\rank (d^*_{\beta_{u_1}})  > N + i$ and
define $u_2 = \max \{s\in S_{u_1} \mid \min\rng (e^*_{\eta_s}P_{I_s}) \leq N\}$.
If $\rank (d^*_{\beta_{u_2}}) + 1 \leq N + i$, then we define
$\eta_{u_1}' = \beta_{u_2}$ and 
$\eta_{u_0}' = (\rank(\eta_{u_1}')+1,\beta_{u_1-},m_{u_0},\varepsilon_{u_1} e^*_{\eta_{u_1}'})$.

We continue inductively for $k\leq i+1$.
If at some point $k$ we have $\max\rng (e^*_{\eta_{u_k}}P_{I_{u_k}}) + k - 1  \leq N + i$, then we 
redefine all $\eta_{u_k},\dots,\eta_{u_0}$ as above and we are done.

If $k=i+1$ and still $\rank (\beta_{u_k}) + k - 1 > N + i$, then
we drop $\eta_{u_k}$ from $\eta_{u_{k-1}}$ redefining 
$\eta_{u_{k-1}}'=(N+1,\beta_{u_k-{}-},m_{u_{k-1}},\varepsilon_{u_k-} e^*_{\eta_{u_k-}})$ (we change only the rank of $\beta_{u_k-}$)
and for all $l=k-2,\dots,0$ we redefine
$\eta_{u_l}' = (\rank(\eta_{u_{l+1}}')+1,\beta_{u_{l+1}-},m_{u_l},\varepsilon_{u_{l+1}} e^*_{\eta_{u_{l+1}}'})$.

We show that for $\gamma'=\eta_\emptyset'$ we have the claimed bound for the norm of the difference.
If we finished the inductive construction without dropping $\eta_{u_{i+1}}$, then 
$e^*_{\gamma}\restriction \langle d_\xi \mid \xi \in \Gamma_N \rangle =  
    e^*_{\gamma'}\restriction \langle d_\xi \mid \xi \in \Gamma_N \rangle$.
If we dropped $\eta_{u_{i+1}}$, then
\[
    \|(e^*_{\gamma}-e^*_{\gamma'})\restriction \langle d_\xi \mid \xi \in \Gamma_N \rangle\| \leq 
        (\prod_{k=0}^i m_{\eta_{u_l}})\|e^*_{\eta_{u_k}}P_{I_{u_k}}\| \leq  3\cdot 4^{-i-1}.
\]
\end{proof}

\begin{lemma}\label{comp}
Let $I\subseteq\N$ be an interval and let $(x_i)_{i\in I}$ be a $C$-RIS with
the growth index $(j_i)_{i\in I}$, $\max\rng x_i + j_i < \min\rng x_{i+1}$, and 
$|d^*_\xi(x_i)| \leq 3Cn_{j_i}^{-1}$ for all $\xi\in\Gamma$ and $i\in I$. 
For every node $\gamma\in\Gamma$ there exist $I' \subset I$, a sequence of intervals $(E_i)_{i\in I'}$, and a node $\gamma'\in \Gamma$ 
with the tree analysis comparable to $(x'_i)_{i\in I'}$, where $x'_i=P_{E_i}x_i$, and
\[
    |e^*_\gamma(\sum_{i\in I} x_i)|\leq 6|e^*_{\gamma'}(\sum_{i\in I'} x'_i)| + 49C4^{-j_1}.
\]
\end{lemma}
\begin{proof}
Fix $\gamma\in\Gamma$ and take $\varepsilon=\pm1$ with 
$|e^*_\gamma(\sum_{i\in I} x_i)| = \varepsilon e^*_\gamma(\sum_{i\in I} x_i)$.
Define $I'=\{i\in I \mid \varepsilon e^*_\gamma(x_i) > 4^{-j_i-1}\}$ and observe that
\begin{equation}\label{comp1}
    \varepsilon e^*_\gamma(\sum_{i\in I\setminus I'} x_i) \leq 4^{-j_1}.
\end{equation}

Let $(I_t, \varepsilon_t, \eta_t)_{t\in\T}$ be the tree analysis of $e^*_\gamma$ and for all $i\in I$
set $t_i^0$ to be the covering index for $x_i$.
For every $i\in I'$ we define $S_i^0 = \{s\in S_{t_i^0} \mid \rng(e^*_{\eta_s}P_{I_s}) \cap \rng x_i \neq \emptyset\}$,
$q_i^0 = \min S_i^0,\ r_i^0 = \max S_i^0$.
Denote the predecessor and the successor of $i$ in $I'$ by $i_-,\,i_+$ respectively.
Define 
\(
   \overline\varepsilon_t=\varepsilon \prod_{\emptyset\preceq u \preceq t} \varepsilon_u,
\)
for every $t\in\T$.

We will construct a $10C$-RIS $(x'_i)_{i\in I'}\subset \bb$ with $x'_i=P_{E_i}x_i$ for some interval $E_i$, such that
\begin{equation}\label{xx'}
    \varepsilon e^*_{\gamma}(x_i) \leq 2 \varepsilon e^*_{\gamma}(x_i').
\end{equation}
    
Fix $i\in I'$.
We consider two cases:

\noindent \textbf{Case A.}
$S_i^0=\{ q_i^0,\, r_i^0 \}$,
$\rng (e^*_{\eta_{q_i^0}}P_{I_{q_i^0}})\cap \rng(x_{i_-})\neq \emptyset$, and
$\rng (e^*_{\eta_{r_i^0}}P_{I_{r_i^0}})\cap \rng(x_{i_+})\neq \emptyset$.

If $\overline\varepsilon_{q_i^0}e^*_{\eta_{q_i^0}}P_{I_{q_i^0}}(x_i)
\geq \overline\varepsilon_{r_i^0}e^*_{\eta_{r_i^0}}P_{I_{r_i^0}}(x_i)$ then 
we define $E_i=\big[0,\min \rng(e^*_{\eta_{r_i^0}}P_{I_{r_i^0}})\big)$ and $x'_i=P_{E_i}x_i$.
We have
\[ 
    \overline\varepsilon_{t^0_i} e^*_{\eta_{t^0_i}}P_{I_{t^0_i}}(x_i) \leq
    2 \overline\varepsilon_{t^0_i} e^*_{\eta_{t^0_i}}P_{I_{t^0_i}}(x_i'),
\]
which implies inequality (\ref{xx'}).

If $\overline\varepsilon_{q_i^0}e^*_{\eta_{q_i^0}}P_{I_{q_i^0}}(x_i) < 
\overline\varepsilon_{r_i^0}e^*_{\eta_{r_i^0}}P_{I_{r_i^0}}(x_i)$  then 
we define $E_i=\big(\max \rng(e^*_{\eta_{q_i^0}}P_{I_{q_i^0}}),\infty\big)$ and $x'_i=P_{E_i}x_i$, so (\ref{xx'}) holds as above.

\noindent \textbf{Case B.} Case A does not hold.
We define $E_i=\rng x_i$, so $x'_i=x_i$ and inequality (\ref{xx'}) holds trivially.

For every $i\in I'$ we define $t_i\succeq t^0_i$ to be the covering index for $x'_i$.
Define $S_i = \{s\in S_{t_i} \mid \rng(e^*_{\eta_s}P_{I_s}) \cap \rng x_i' \neq \emptyset\}$,
$q_i = \min S_i,\ r_i = \max S_i$.

Observe that 
\[ \tag{$T1$} 
\text{for every $i\in I'$ there exists $s\in S_i$ with  $\rng(e^*_{\eta_{s}}P_{I_s})\subset \rng x'_i$.} 
\]
Indeed, since $t_i$ is the covering index for $x_i'$, then $q_i\neq r_i$.
In Case B its assumption yields exactly the existence of the required $s\in S_i^0=S_i$.
In Case A the definition of $E_i$ yields that $\rng (e^*_{\eta_{q_i}}P_{I_{q_i}})\cap \rng(x_{i_-})= \emptyset$ or
$\rng (e^*_{\eta_{r_i}}P_{I_{r_i}})\cap \rng(x_{i_+}) = \emptyset$, hence $(T1)$ holds.

Using Lemma \ref{lenpath} for $x'_i,\ j_i+1,\ t_i$ (the lower bound follows from (2)) for every $i\in I'$ we get that
\[ \tag{$T2$}
  \text{the length of the branch linking the root and the node $t_i$ is not greater than $j_i+2$}
\]

Let $\T^0\subseteq\T$ be the smallest subtree containing the root and all $t_i$ for $i\in I'$.
For every $t\in \T^0$ define $i_t=\max \{i\in I\mid m_{j_i} \leq m_t\}$ or $i_t=1$,
if the set is empty.

Fix $i\in I$ and notice that if there exists $t\preceq t_i$ such that $i<i_t$ then
we have $m_{j_i}<m_{j_{i_t}}\leq m_t$, hence by Lemma \ref{gonRIS}
\[ 
    |e^*_{\eta_t}P_{I_t}(x'_i)| \leq 6Cm_t^{-1} \leq m_{j_i}^{-1}.
\]

We show that $i\not\in I'$.
Consider the following formula
\[
    \varepsilon e^*_{\gamma}(x'_i) = \overline\varepsilon_t e^*_{\eta_t}P_{I_t}(x'_i) 
+ \sum_{u\prec t} \big( \prod_{\emptyset\preceq v\prec u}m_v^{-1}\big) \bdp(e^*_{\eta_u})P_{I_u}(x'_i).
\]
Since $t_i$ is the covering index then for every
$u\prec t$ there are at most two functionals in 
$\bdp(e^*_{\eta_u})$ with non-zero action on $x'_i$.
By $(T2)$ we get
\[ 
    |e^*_{\gamma}(x'_i)| \leq m_{j_i}^{-1} + 
    (j_i+2) 2 \cdot 3Cn_{j_i}^{-1} < 4^{-j_i},
\]
hence $i\not\in I'$.
Therefore, 
\[\tag{$T3$}
    \text{for every $i\in I'$ and every $t\preceq t_i$ we have $i\geq i_t$.}
\]


    

First, we will define $\gamma'\in\Gamma$ such that $e^*_{\gamma'}$ is comparable with $(x'_i)_{i\in I'}$ satisfying certain estimates.
Second, we will show the claimed inequality.

\noindent \textbf{STEP 1}. Inductive construction.
We define inductively on $I'$ a sequence $(\gamma^i)_{i\in \{0\}\cup I'}\subset \Gamma$ with $\gamma^0=\gamma$ and a sequence
$(\T^i)_{i\in I'}$ with $\T^i\subseteq\T^0$. 
Fix $i\in I'$ and denote $(\min I')_-=0$.

On the $i$-th step we change $\eta_s^{i_-}$ for some $s\in S_i$ to obtain $\gamma^i\in\Gamma$ 
(by Remark \ref{mt-part}) with
tree analysis $(I_t^i, \varepsilon_t, \eta_t^i)_{t\in\T^i}$ satisfying the following 
conditions
\begin{enumerate}[(i)]
    \item $e^*_{\gamma^i}$ is comparable with $(x'_j)_{j\leq i,\, j\in I'}$,
    \item for all $j>i,\,j\in I'$ we have 
        $\mt(e^*_{\eta_{t_i}^{i_-}})P_{I_{t_i}^{i_-}}(x_j') = \mt(e^*_{\eta_{t_i}^i})P_{I_{t_i}^i}(x_j')$,
    \item   
        if $i>i_{t_i}$ then
        \(
            \overline\varepsilon_{t_i} \mt(e^*_{\eta_{t_i}^{i_-}})P_{I_{t_i}^{i_-}}(x'_i) \leq 
            3 \overline\varepsilon_{t_i} \mt(e^*_{\eta_{t_i}^i})P_{I_{t_i}^i} (x'_i) +
            3C4^{-j_{i}},
        \)
        
    \noindent if $i=i_{t_i}$ then
    \(
        \overline\varepsilon_{t_i} e^*_{\eta_{t_i}^{i_-}}P_{I_{t_i}^{i_-}}(x'_i) \leq 
            3 \overline\varepsilon_{t_i} e^*_{\eta_{t_i}^i}P_{I_{t_i}^i} (x'_i) + 3C4^{-j_{i}},
    \) 
    \item  if $i>i_{t_i}$ then
    \(
        \overline\varepsilon_{t_i} e^*_{\eta_{t_i}^{i_-}}P_{I_{t_i}^{i_-}}(\sum_{j<i,\,j\in I'} x'_j) \leq
        \overline\varepsilon_{t_i} e^*_{\eta_{t_i}^i}P_{I_{t_i}^i}(\sum_{j<i,\,j\in I'} x'_j)+ 3C4^{-j_{i_-}},
    \)
    
    \noindent if $i=i_{t_i}$ then
        \(
            e^*_{\eta_{t_i}^i}P_{I_{t_i}^i}(x'_j) = 0, \text{ for all } j<i,\,j\in I',
        \)
    \item $\T^i\subseteq\T^{i_-}$,
    \item for all $t\in \T^i$ incomparable with $t_i$ we have $\eta^i_t=\eta^{i_-}_t,\ I^i_t = I^{i_-}_t$.
    
\end{enumerate}
We do the induction on $I'$.
Note that the induction stabilises after a finite number of steps since $e^*_\gamma$ has a finite range.
The inductive base is an easier version of the inductive step (there are no changes to be made to the left of 
$\rng x'_{\min I'}$), so we show only the latter.
Fix $i\in I'$.
We assume the inductive hypothesis for all $j<i,\,j\in I'$.

In order to get comparability we change $\eta_{t_i}^{i_-}$.
We treat possibilities $i>i_{t_i}$ and $i=i_{t_i}$ separately (recall that by $(T3)$ there are no $i<i_{t_i}$).

\noindent \textbf{Case I.} $i>i_{t_i}$.
We define $\T^i=\T^{i_-}$.
By $(T1)$ there exists $s\in S_i$ with 
$\rng(e^*_{\eta_{s}^{i_-}}P_{I_{s}^{i_-}})\subset \rng x'_i$. 

On the right end of $\rng x'_i$ we have the following cases:

\noindent \textbf{a)} $\rng (e^*_{\eta_{r_i}^{i_-}}P_{I_{r_i}^{i_-}})\cap \rng(x'_{i_+})\neq \emptyset$ and
$\overline\varepsilon_{r_i-}e^*_{\eta_{r_i-}^{i_-}}P_{I_{r_i-}^{i_-}}(x'_i)\geq 
\overline\varepsilon_{r_i}e^*_{\eta_{r_i}^{i_-}}P_{I_{r_i}^{i_-}}(x'_i)$.

Define $\eta_{r_i-}^i = \eta_{r_i-}^{i_-}$, $I_{r_i-}^i= I_{r_i-}^{i_-}$,
$\eta_{r_i}^i= \eta_{r_i}^{i_-}$, and $I_{r_i}^i = I_{r_i}^{i_-}\cap [\min\rng x'_{i_+},\infty)$.

\noindent \textbf{b)} $\rng (e^*_{\eta_{r_i}^{i_-}}P_{I_{r_i}^{i_-}})\cap \rng(x'_{i_+})\neq \emptyset$ and
$\overline\varepsilon_{r_i-}e^*_{\eta_{r_i-}^{i_-}}P_{I_{r_i-}^{i_-}}(x'_i) <
\overline\varepsilon_{r_i}e^*_{\eta_{r_i}^{i_-}}P_{I_{r_i}^{i_-}}(x'_i)$.

Proposition \ref{split} for $(\eta_{r_i}^{i_-}, \max\rng x'_i, j_i)$ gives a suitable node
$\eta_{r_i-}^i\in\Gamma$ 
and we define $I_{r_i-}^i=[\min \rng (e^*_{\eta_{r_i}^{i_-}}P_{I_{r_i}^{i_-}}),\, \rank(\eta_{r_i-}^i)]$,
$\eta_{r_i}^i = \eta_{r_i}^{i_-}$ and $I_{r_i}^i=I_{r_i}^{i_-}\cap [\min\rng x'_{i_+},\infty)$.

\noindent \textbf{c)} $\rng (e^*_{\eta_{r_i}^{i_-}}P_{I_{r_i}^{i_-}})\cap \rng(x'_{i_+}) = \emptyset$.

Define $\eta_{r_i}^i = \eta_{r_i}^{i_-}$ and $I_{r_i}^i=I_{r_i}^{i_-}$.

On the left end of $\rng x'_i$ we have the following cases:

\noindent \textbf{d)} $\rng (e^*_{\eta_{q_i}^{i_-}}P_{I_{q_i}^{i_-}})\cap \rng(x'_{i_-})\neq \emptyset$ and
$\overline\varepsilon_{q_i+}e^*_{\eta_{q_i+}^{i_-}}P_{I_{q_i+}^{i_-}}(x'_i)\geq 
\overline\varepsilon_{q_i}e^*_{\eta_{q_i}^{i_-}}P_{I_{q_i}^{i_-}}(x'_i)$.

Define $\eta_{q_i+}^i = \eta_{q_i+}^{i_-}$, $I_{q_i+}^i= I_{q_i+}^{i_-}$, and $I_{q_i}^i=I_{q_i}^{i_-}\cap [0,\rank(\eta_{q_i}^i)]$,
where $\eta_{q_i}^i$ is given by Proposition \ref{split} for \((\eta_{q_i}^{i_-}, \max\rng x'_{i_-}, j_{i_-})\),

\noindent \textbf{e)} $\rng (e^*_{\eta_{q_i}^{i_-}}P_{I_{q_i}^{i_-}})\cap \rng(x'_{i_-})\neq \emptyset$ and
$\overline\varepsilon_{q_i+}e^*_{\eta_{q_i+}^{i_-}}P_{I_{q_i+}^{i_-}}(x'_i) < 
\overline\varepsilon_{q_i}e^*_{\eta_{q_i}^{i_-}}P_{I_{q_i}^{i_-}}(x'_i)$.

Proposition \ref{split} for $(\eta_{q_i}^{i_-}, \max\rng x'_{i_-}, j_{i_-})$ gives a suitable node
$\eta_{q_i}^i\in\Gamma$ 
and we define $I_{q_i}^i=[\min \rng (e^*_{\eta_{q_i}^{i_-}}P_{I_{q_i}^{i_-}}),\, \rank(\eta_{q_i}^i)]$,
$\eta_{q_i+}^i = \eta_{q_i}^{i_-}$ and $I_{q_i+}^i=I_{q_i}^{i_-}\cap [\min\rng x'_i,\infty)$.

\noindent \textbf{f)} $\rng (e^*_{\eta_{q_i}^{i_-}}P_{I_{q_i}^{i_-}})\cap \rng(x'_{i_-}) = \emptyset$.

Define $\eta_{q_i}^i = \eta_{q_i}^{i_-}$ and $I_{q_i}^i=I_{q_i}^{i_-}$.

After introducing the changes listed above the conditions (i)-(v) are satisfied.







\noindent \textbf{Case II.} $i=i_{t_i}$.
We define $I_{t_i}^i=I_{t_i}^{i_-}\cap [\min\rng x'_i,\infty)$, so we don't have to change $\eta_{q_i}^{i_-}$.
Then we change $\eta_{r_i}^{i_-}$ as in case 1a), 1b) or 1c).
We define 
\[
    \T^i=\T^{i_-}\setminus\{\,s\in \T^{i_-}\mid s \text{ is a successor of $t_i$ in } \T^{i_-},\ \max\rng(e^*_{\eta_s}P_{I_s})<
    \min\rng x_i'\,\}.
\]
The conditions (i)-(v) are satisfied.

For any $s\in S_i$ not considered above we set $\eta^i_s=\eta^{i_-}_s$ and $I^i_s = I^{i_-}_s$.
By Remark \ref{mt-part} there is a node $\gamma^i$ with the tree analysis $(I_t^i, \varepsilon_t, \eta_t^i)_{t\in\T^i}$, where 
$\eta^i_t=\eta^{i_-}_t$ and $I^i_t = I^{i_-}_t$ for every $t\in\T^i$ incomparable with $t_i$, i.e. we obtain (vi).

This finishes the inductive construction.
Notice that change $\gamma^{i_-}\to\gamma^i$ induces only the following changes  in the action on $\sum_j x'_j$: $e^*_{\eta^{i_-}_t}\to e^*_{\eta^i_t}$ on $\rng x'_i$
and $\bdp(e^*_{\eta^{i_-}_t})\to\bdp(e^*_{\eta^i_t})$ for $t\preceq t_i$ on $[\min\rng x'_i,\infty)$.

We define $\gamma'$ to be the node on which the inductive construction stabilises.
Similarly, we define $\T'$.
This finishes STEP 1.
We proceed to show the estimate-part of the lemma.

\noindent \textbf{STEP 2}. Estimation of the errors.
Fix $i\in I'$.
We show that after introducing changes on the $i$-th step in the induction we have the following:
\begin{equation}\label{compi}
    \varepsilon e^*_{\gamma^{i_-}}(x'_i) \leq 3 \varepsilon e^*_{\gamma^i}(x'_i) + 4C4^{-j_i}.
\end{equation}
\begin{equation}\label{compjl}
    \varepsilon e^*_{\gamma^{i_-}}(\sum_{j<i,\,j\in I'}x'_j) \leq \varepsilon e^*_{\gamma^i}(\sum_{j<i,\,j\in I'}x'_j) + 3C4^{-j_i}.
\end{equation}
\begin{equation}\label{compjg}
    \varepsilon e^*_{\gamma^{i_-}}(x'_j) \leq \varepsilon e^*_{\gamma^i}(x'_j) + j_jCm_{j_j}^{-1}\text{ for ever } j>i,\, j\in I'.
\end{equation}

\noindent \textbf{Ad (\ref{compi})}. 
We use the equalities
\[
    \varepsilon e^*_{\gamma^{i_-}}(x'_i) = \overline\varepsilon_{t_i} e^*_{\eta_{t_i}^{i_-}}P_{I_{t_i}^{i_-}}(x'_i) 
+ \sum_{t\prec t_i} \big( \prod_{\emptyset\preceq u\prec t}m_u^{-1}\big) \bdp(e^*_{\eta_t^{i_-}})P_{I_t^{i_-}}(x'_i),
\]
\[
   \varepsilon e^*_{\gamma^i}(x'_i) = \overline\varepsilon_{t_i} e^*_{\eta_{t_i}^i}P_{I_{t_i^i}}(x'_i) 
+ \sum_{t\prec t_i} \big( \prod_{\emptyset\preceq u\prec t}m_u^{-1}\big)\bdp(e^*_{\eta_t^i})P_{I_t}^i(x'_i).
\]
We start with estimating bd-parts on $x'_i$, i.e. $\sum_{t\prec t_i} \big( \prod_{\emptyset\preceq u\prec t}m_u^{-1}\big)\bdp(e^*_{\eta_t^{i_-}})P_{I_t^{i_-}}(x'_i)$ and 
$\sum_{t\prec t_i} \big( \prod_{\emptyset\preceq u\prec t}m_u^{-1}\big) \bdp(e^*_{\eta_t^i})P_{I_t^i}(x'_i)$.
Since $t_i$ is a covering index for $x'_i$ there are at most two elements in bd-part of $e^*_{\eta_t^{i_-}}$ with
non-zero action on $x'_i$ for every $t\prec t_i$.
Using $(T2)$ we get
\begin{equation}\label{combd}
    \sum_{t\prec t_i}\big( \prod_{\emptyset\preceq u\prec t}m_u^{-1}\big) \bdp(e^*_{\eta_t^{i_-}})P_{I_t^{i_-}}(x'_i) \leq (j_i+2) 2\cdot3Cn_{j_i}^{-1} < Cm^{-1}_{j_i}.
\end{equation}
Similarly
\begin{equation}\label{combd'}
    \sum_{t\prec t_i}\big( \prod_{\emptyset\preceq u\prec t}m_u^{-1}\big) \bdp(e^*_{\eta_t^i})P_{I_t^i}(x'_i) \leq (j_i+2) 2\cdot3Cn_{j_i}^{-1} < Cm^{-1}_{j_i}.
\end{equation}

If $i=i_{t_i}$, then using (iii), (\ref{combd}), and (\ref{combd'}) we get (\ref{compi}). 
If $i>i_{t_i}$, then we observe that for the bd-part we have
\begin{equation}\label{compb}
    \overline\varepsilon_{t_i}\bdp(e^*_{\eta_{t_i}^{i_-}})P_{I_{t_i}^{i_-}}(x'_i) \leq
    n_{t_i}3Cn_{j_i}^{-1} \leq Cm^{-1}_{j_i}, \ 
    \overline\varepsilon_{t_i}\bdp(e^*_{\eta_{t_i}^i})P_{I_{t_i}^i}(x'_i) \leq
    n_{t_i}3Cn_{j_i}^{-1} \leq Cm^{-1}_{j_i}.
\end{equation}
By (iii), (\ref{combd}), (\ref{combd'}), and (\ref{compb}) we get (\ref{compi}).


    
\noindent \textbf{Ad (\ref{compjl})}. If $i>i_{t_i}$ then (\ref{compjl})
follows directly from (iv) and (vi).
If $i=i_{t_i}$, then using Lemma \ref{gonRIS} for every $j<i$ we obtain
\[
    \overline\varepsilon_{t_i}e^*_{\eta_{t_i}^{i_-}}P_{I_{t_i}^{i_-}}(\sum_{j < i,\,j\in I'} x'_j) \leq (i-1)6Cm_{t_i}^{-1} 
    \leq (i-1)6Cm_{j_i}^{-1} \leq C4^{-j_i}.
\]
This, (iv) and (vi) gives (\ref{compjl}).



\noindent \textbf{Ad (\ref{compjg})}.
Fix $j>i,\, j\in I'$.
We prove inductively that for every $t\preceq t_i$ we have
\[
    \overline\varepsilon_te^*_{\eta_t^{i_-}}P_{I_t^{i_-}}(x'_j) \leq 
    \overline\varepsilon_te^*_{\eta_t^i}P_{I_t^i}(x'_j) + 
    \#\{u \in \T^i \mid t\preceq u \preceq t_i \}\cdot Cm_{j_j}^{-1}.
\]
Start from $t=t_i$.
The condition (ii) gives
\[
    \overline\varepsilon_{t_i}e^*_{\eta_{t_i}^{i_-}}P_{I_{t_i}^{i_-}}(x'_j) = 
    \overline\varepsilon_{t_i}e^*_{\eta_{t_i}^i}P_{I_{t_i}^i}(x'_j) + 
    \overline\varepsilon_{t_i}\sum_{s\in S_{t_i}} d^*_{\beta_s^{i_-}}(x'_j) - 
    \overline\varepsilon_{t_i}\sum_{s\in S_{t_i}} d^*_{\beta_s^i}(x'_j)
\]
By ($T3$) we have $i_t\leq i < j$, and thus $n_t < n_{j_j}$.
It follows that 
\[
    \overline\varepsilon_t\sum_{s\in S_t} d^*_{\beta_s^{i_-}}(x'_j) - 
    \overline\varepsilon_t\sum_{s\in S_t} d^*_{\beta_s^i}(x'_j)
    \leq 2n_t3Cn_{j_j}^{-1} \leq Cm_{j_j}^{-1},
\]
which finishes the inductive base.

Let $t\neq t_i$.
For every $s\in S_t$ with $s \not\preceq t_i$ we have  by (vi) 
$\overline\varepsilon_se^*_{\eta_s^{i_-}}P_{I_s^{i_-}}(x'_j) = \overline\varepsilon_se^*_{\eta_s^i}P_{I_s^i}(x'_j)$.
Moreover, for $s_0\in S_t$ with $s_0 \preceq t_i$ we have the inductive hypothesis, so
\[
    \overline\varepsilon_t\mt(e^*_{\eta_t^{i_-}})P_{I_t^{i_-}}(x'_j) \leq 
    \overline\varepsilon_t\mt(e^*_{\eta_t^i}P_{I_t^i}(x'_j)) + 
    \#\{u \in \T^i \mid s_0\preceq u \preceq t_i \}\cdot Cm_{j_j}^{-1}.
\]
As in the inductive base we have for bd-parts 
\[
    |\sum_{s\in S_t} d^*_{\beta_s^{i_-}}(x'_j)|, \ 
    |\sum_{s\in S_t} d^*_{\beta_s^i}(x'_j)|
    \leq n_t3Cn_{j_j}^{-1} \leq Cm_{j_j}^{-1}/2,
\]
hence
\begin{align*}
    \overline\varepsilon_te^*_{\eta_t^{i_-}}P_{I_t^{i_-}}(x'_j) 
    & = \overline\varepsilon_t\sum_{s\in S_t} d^*_{\beta_s^{i_-}}(x'_j) + \overline\varepsilon_t\mt(e^*_{\eta_t^{i_-}})P_{I_t^{i_-}}(x'_j) \\
    & \leq  n_t3Cn_{j_j}^{-1} + \overline\varepsilon_t\mt(e^*_{\eta_t^i}P_{I_t^i}(x'_j)) + 
    \#\{u \in \T^i \mid s_0\preceq u \preceq t_i \}\cdot Cm_{j_j}^{-1} \\
    & = n_t3Cn_{j_j}^{-1} + \overline\varepsilon_te^*_{\eta_t^i}P_{I_t^i}(x'_j) - \overline\varepsilon_t\sum_{s\in S_t} d^*_{\beta_s^i}(x'_j) +
    \#\{u \in \T^i \mid s_0\preceq u \preceq t_i \}\cdot Cm_{j_j}^{-1} \\
    & \leq \overline\varepsilon_te^*_{\eta_t^i}P_{I_t^i}(x'_j) + \#\{u \in \T^i \mid t\preceq u \preceq t_i \}\cdot Cm_{j_j}^{-1}.
\end{align*}
The inductive step is finished.

Using $(T2)$ we obtain (\ref{compjg}). Indeed,
\[
    \varepsilon e^*_{\gamma^{i_-}}(x'_j) 
        \leq \varepsilon e^*_{\gamma^i}(x'_j) + (j_i+2)Cm_{j_j}^{-1}
        \leq \varepsilon e^*_{\gamma^i}(x'_j) + j_jCm_{j_j}^{-1}.
\]

Finally, having (3), (4) and (5) we obtain by an easy induction that for every $i\in I'$ we have
\[
    \varepsilon e^*_\gamma(\sum_{j\leq i ,\, j\in I'} x'_j) \leq 3 \varepsilon e^*_{\gamma^i}(\sum_{j\leq i ,\, j\in I'} x'_j) + 23C\sum_{j\leq i ,\, j\in I'}4^{-j_j},
\]
\[
    \varepsilon e^*_\gamma( x'_j) \leq \varepsilon e^*_{\gamma^i}(x'_j) + ij_jm_{j_j}^{-1}, \text{ for all } j>i,\, j\in I'.
\]

This, (1), and (2) finishes the proof of the lemma.
\end{proof}

\begin{lemma}\label{norm-sum}
Let $(y_i)_{i=1}^{n_j}$ be a  $C$-RIS in $\bb$ such that every $y_i$ is a projection on some interval of a $C-\ell^{n_{j_i}}_1$-average and let $(z_i)_{i=1}^{n_j}$ be a subsequence of
the standard basis of $\mT$. 
Let $m_{j_1} \geq 2n_j^3$, $j\geq 5$, and assume that $\max\rng y_i+j_i<\min\rng y_{i+1}$ for all $i$'s.
Then, for every interval $J$, there exists a functional $f$ in the norming set $W$ such that $\supp f \subseteq\supp(\sum_{j\in J}z_j)$ and 
\[
    \|\sum_{i\in J}y_i\| \leq 96f(\sum_{i\in J}z_i).
\]
\end{lemma}

\begin{proof}
Set $y=\sum_{i\in J}y_i$ and $z=\sum_{i\in J}z_i$ and 
choose $\gamma\in\Gamma$ with $|e^*_\gamma(y)|\geq \frac{6}{7}\|y\|$ and
the tree analysis $(I_t, \varepsilon_t, \eta_t)_{t\in\T}$.
By Lemma \ref{comp} we can assume that the tree analysis of $e^*_\gamma$ is comparable with $(y'_i)_{i\in J'}$, where $y'_i$ is a projection of $y_i$ on some interval,
and $|e^*_\gamma(\sum_{i\in J'} y'_i)|\geq \frac{1}{7}\|y\| - 9C4^{-j_1}$ for some $J'\subseteq J$.
We will construct $f\in W$ such that $\supp f\subseteq \bigcup_{i\in J'}\supp z_i$ and
$|e^*_\gamma(\sum_{i\in J'} y'_i)| \leq 12f(\sum_{i\in J'} z_i) = 12f(z)$,
and thus we may assume that $J'=J$.
Moreover, by Lemma \ref{gonRIS} the sequence $(y'_i)_{i\in J}$ is $10C$-RIS so we may assume that $y'_i=y_i$.

For every $i$ let $t_i\in \T$ be the covering index for $y_i$
and for every $t\in \T$ let $E_t=\{i\in J\mid t=t_i\}$.
\begin{claim}\label{x}
Fix $t=t_{i_0}$ for some $i_0$ .
\begin{enumerate}
    \item[(a)] Assume that $m_t<m_{j_1}$.
    Then for $g_i = z_i^*$, $i\in E_t$, we have
    \[
        |e^*_{\eta_t}P_{I_t}(\sum_{i\in E_t}y_i)| \leq \frac{1}{n_j^2} + 
        \frac{6}{m_t}\sum_{i\in E_t}g_i(z).
    \]
    \item[(b)] Assume that $m_t \geq m_{j_1}$.
    Then there exists $g_t\in W$ with 
        $\supp g_t \subseteq \supp (\sum_{i\in E_t} z_i)$, and such that
    \[
         |e^*_{\eta_t}P_{I_t}(\sum_{i\in E_t}y_i)| \leq \frac{1}{n_j^2} + 
        12g_t(z).
    \]
\end{enumerate}
\end{claim}
In (a) by defining $g_t = \frac{1}{m_t}\sum_{i\in E_t}g_i$, we obtain the inequality of $(b)$, 
however, in the sequel we need the more precise form, hence we keep in $(a)$ the precise form of $g_t$.
\begin{proof}[Proof of the claim]
(a).
We set $S_i=\{\,s\in S_t \mid \rng(e^*_{\eta_s}P_{I_s}) \subset \rng(y_i)\,\}$ and 
set $g_i = z_i^*$, $i\in E_t$.
Using Lemmas \ref{average-bound}(2) and \ref{norm-proj} for every $y_i$ (see Remark \ref{aveproj}) we estimate
\begin{align*}
    |e^*_{\eta_t}P_{I_t}(\sum_{i\in E_t} y_i)| 
    & = |\sum_{s\in S_t}d^*_{\beta_s}(\sum_{i\in E_t} y_i) + 
        \frac{1}{m_t} \sum_{i\in E_t} \sum_{s\in S_i}e^*_{\eta_s}P_{I_s}(y_i)| \\
    & \leq n_t \frac{3C}{n_{j_1}} + 
        \frac{1}{m_t} 3 \sum_{i\in E_t}  \left(1+\frac{2|S_i|}{n_{j_i}}\right) \\
    & \leq \frac{1}{n_j^2} +
        \frac{1}{m_t} 6\sum_{i\in E_t} g_i(z).
\end{align*}

(b). Define $s_t=\min S_t$, and
$i_t = \max\{i\in E_t\mid m_{j_i} \leq m_t\}$.
For every $i<i_t$ we have by Lemma \ref{gonRIS}
\[
    |e^*_{\eta_t}P_{I_t}(y_i)| \leq \frac{6C}{m_t} \leq \frac{6C}{m_{i_t}} < \frac{1}{n_j^3}.
\]
This implies
\begin{equation}\label{b1}
    |e^*_{\eta_t}P_{I_t}(\sum_{i\in E_t,\, i<i_t} y_i)| \leq \frac{1}{2n_j^2}.
\end{equation}

Now let $i\in E_t,\ i \geq i_t$.
If 
\begin{equation}\label{it>}
    |e^*_{\eta_t}P_{I_t}(y_{i_t})| > 
    |e^*_{\eta_t}P_{I_t}(\sum_{i\in E_t,\, i > i_t} y_i)|,
\end{equation}
then for $g_t = z_{i_t}^*$ we obtain
\begin{equation}\label{b2}
    |e^*_{\eta_t}P_{I_t}(\sum_{i\in E_t,\, i \geq i_t} y_i)| \leq 2\cdot 3C g_t(z).
\end{equation}
On the other hand, if (\ref{it>}) is not true, 
then we proceed as in Claim 1.
We set $S_i=\{\,s\in S_t \mid \rng(e^*_{\eta_s}P_{I_s}) \subset \rng(y_i)\,\}$ and 
 $g_i = z_i^*$.
Then we set $g_t=\frac{1}{m_t} \sum_{i\in E_t,\, i > i_t} g_i$.
The definition is correct as $g_i$'s are pairwise different
and $\# E_t \leq \# S_t \leq n_t$.
Using Lemmas \ref{average-bound}(2) and \ref{norm-proj} for every $y_i$  (see Remark \ref{aveproj}) we estimate
\begin{align} \label{b3}
    |e^*_{\eta_t}P_{I_t}(\sum_{i\in E_t,\, i \geq i_t} y_i)| 
    & \leq 2|\sum_{s\in S_t}d^*_{\beta_s}(\sum_{i\in E_t,\, i > i_t} y_i) + 
        \frac{1}{m_t} \sum_{i\in E_t,\, i > i_t} \sum_{s\in S_i}e^*_{\eta_s}P_{I_s}(y_i)| \\ \nonumber
    & \leq 2n_t \frac{3C}{n_{j_{i_t+1}}} + 
        2\frac{1}{m_t} 3 \sum_{i\in E_t,\, i > i_t}  \left(1+\frac{2|S_i|}{n_{j_i}}\right) \\ \nonumber
    & \leq \frac{1}{2n_j^2} +
        \frac{1}{m_t} 12\sum_{i\in E_t,\, i > i_t} g_i(z) \\ \nonumber
    & \leq \frac{1}{2n_j^2} + 12g_t(z).
\end{align}
Combining (\ref{b1}), (\ref{b2}), and (\ref{b3}) we obtain (b).
\end{proof}

Let $\T'$ be the smallest subtree of $\T$ containing the root and all $t_i$'s. 
\begin{claim}\label{y}
For every $t\in \T'$ there exists $f_t\in W$ such that 
$\rng f_t \subseteq \bigcup \{\rng z_i \mid \rng y_i \cap
\rng(e^*_{\eta_t}P_{I_t}) \neq 0 \}$ and 
\[
    |e^*_{\eta_t}P_{I_t}(y)| \leq  
    \frac{3}{n_j^2}\#\{s \in \T'\mid s\succeq t\} + 12f_t(z).
\]
\end{claim}
\begin{proof}[Proof of the claim.]
We prove the claim by induction on the tree $\T'$ starting from terminal nodes.

Let $t$ be a terminal node. 
Then $t=t_{i_0}$ for some $i_0$ and we use Claim \ref{x}.
If $m_t < m_{j_1}$, then Claim \ref{x}(a) gives $(g_i)_{i\in E_t}$,
we set $f_t=\frac{1}{m_t}\sum_{i\in E_t}g_i$ and we are done.
If $m_t \geq m_{j_1}$, then we just take $f_t=g_t$ given by Claim \ref{x}(b).

Let $t$ be a non-terminal node and define $R_t = J\setminus E_t$. 
Observe that for every $i\in R_t$ we have $t\prec t_i$, so there are at most two
$s\in S_t$ such that $d^*_{\beta_s}(y_i) \neq 0$.
This implies
\begin{equation}\label{bullet}
    |\sum_{s\in S_t}d^*_{\beta_s}(\sum_{i\in R_t} y_i)|
    \leq 2n_j\frac{3C}{n_{j_1}} \leq \frac{1}{n_j^2}.
\end{equation}
The inductive assumption gives for every $s\in S_t\cap\T'$ that
\[
    |e^*_{\eta_s}P_{I_s}(y)| \leq 
    \frac{3}{n_j^2}\#\{u\in \T' \mid u \succeq s\}) + 12f_s(z),
\]
where $f_s\in W$ is such that 
$\supp f_s \subseteq \bigcup \{\supp z_i \mid \rng y_i \cap
\rng(e^*_{\eta_s}P_{I_s}) \neq 0 \}$.

If $t\neq t_i$ for all $i$ then we define 
$f_t=\frac{1}{m_t}\sum_{s\in S_t\cap \T'} f_s$.
Functional $f_t$ is in $W$ since by comparability all $(f_s)_{s\in S_t\cap \T'}$ have pairwise disjoint ranges.
We have $E_t = \emptyset$, hence by (\ref{bullet}) and the inductive assumption
\begin{align*}
    |e^*_{\eta_t}P_{I_t}(y)| & 
    \leq |\sum_{s\in S_t}d^*_{\beta_s}(\sum_{i\in R_t} y_i)|
    + |\frac{1}{m_t}\sum_{s\in S_t\cap \T'}e^*_{\eta_s}P_{I_s}(y)| \\
    & 
    \leq \frac{3}{n_j^2}(1+\sum_{s\in S_t\cap \T'}\# 
    \{u\in \T' \mid u \succeq s\}) + 12f_t(z),
\end{align*}
and we are done.

If $t=t_{i_0}$ for some $i_0$ then we use Claim \ref{x}. 
We have the following cases.

If $m_t < m_{j_1}$ then from Claim $1(a)$ we have 
\[
        |e^*_{\eta_t}P_{I_t}(\sum_{i\in E_t}y_i)| \leq
        \frac{1}{n_j^2} + 
        \frac{6}{m_t}\sum_{i\in E_t}g_i(z).
\]
Define 
\[
    f_t=\frac{1}{m_t}\left( \sum_{s\in S_t\cap\T'} f_s
        +\sum_{i\in E_t}g_i\right).
\]
Observe that $f_t\in W$ since 
$\{\rng f_s\mid s\in S_t\cap\T'\} \cup \{\rng g_i \mid i\in E_t\}$ 
are pairwise disjoint and $\# E_t \leq \#(S_t\setminus\T')$.
Therefore, using (\ref{bullet}) and the inductive assumption
\begin{align*}
    |e^*_{\eta_t}P_{I_t}(\sum_{i\in R_t} y_i) |
    & \leq |\sum_{s\in S_t}d^*_{\beta_s}(\sum_{i\in R_t} y_i)| + 
        \frac{1}{m_t} |\sum_{s\in S_t\cap\T'}e^*_{\eta_s}P_{I_s}
        (\sum_{i\in R_t}y_i)|  \\
    & \leq 
        \frac{1}{n_j^2} + \frac{1}{m_t} \sum_{s\in S_t\cap \T'}
        \left(\frac{3}{n_j^2}\#\{u\in \T'\mid u\succeq s\} + 12f_s(z)\right),
\end{align*}
and thus
\[
|e^*_{\eta_t}P_{I_t}(y)| \leq 
|e^*_{\eta_t}P_{I_t}(\sum_{i\in R_t} y_i)| + 
|e^*_{\eta_t}P_{I_t}(\sum_{i\in E_t} y_i)|
\leq 
        \frac{3}{n_j^2}(1 + \sum_{s\in S_t\cap \T'}\#\{t_i\mid t_i\succeq s\})
         + 12f_t(z),
\]
and we are done.

We are left with the case $m_t \geq m_{j_1}$.
Claim $1(b)$ gives $g_t\in W$ such that 
$\supp g_t \subseteq \supp (\sum_{i\in E_t} z_i)$, and
\[
    |e^*_{\eta_t}P_{I_t}(\sum_{i\in E_t}y_i)| \leq \frac{1}{n_j^2} + 
    12g_t(z),
\]
whereas (\ref{bullet}) gives, as $s\in S_t,\ i\in R_t,\ s\preceq t_i$, that
\begin{align*}
     |e^*_{\eta_t}P_{I_t}(\sum_{i\in R_t} y_i)| & \leq
     |\sum_{s\in S_t}d^*_{\beta_s}(\sum_{i\in R_t} y_i)| + 
        \frac{1}{m_t} |\sum_{s\in S_t}e^*_{\eta_s}P_{I_s}
        (\sum_{i\in R_t}y_i)| \\
    & \leq \frac{1}{n_j^2} + \frac{1}{m_t}\sum_{i\in R_t}\|y_i\| 
     \leq \frac{1}{n_j^2} + \frac{n_j}{m_t} 
     \leq  \frac{2}{n_j^2}.
\end{align*}
Combining the two above estimations and setting $f_t=g_t$ 
finishes the inductive proof.
\end{proof}

Let $f=f_\emptyset \in W$.
Claim \ref{y} yields
\[
    \frac{1}{7}\|y\| -9C4^{-j_1}\leq |e^*_\gamma(y)| \leq 
    \frac{3}{n_j^2}\#\{t\leq \T'\mid t\succeq \emptyset\}+12f(z)
    \leq \frac{6}{n_j} + 12f(z),
\]
as $\|y\|\geq 1$ implies
\[
    \|y\| \leq 8\cdot12f(z) \leq 96f(z).
\]
The proof of the lemma is finished.
\end{proof}

\section{Proof of Main Theorem.}

\begin{theorem}\label{thm:a}
For every infinite dimensional block subspace $Y$ of $\bb$ there exists
a block sequence $(y_i)_{i\in \N}$ in $Y$ and sequences of natural numbers $(j_i)_{i\in\N}$, 
$((k_{i,j})_{j=1}^{n_{j_i}})_{i\in\N}$ such that sequences 
$(y_i)_{i\in \N}$ and 
$(\frac{m_{j_i}}{n_{j_i}}\sum_{j=1}^{n_{j_i}} e_{k_{i,j}} )_{i\in\N}\subset \mT$ are equivalent.
\end{theorem}

\begin{proof}

Fix some subspace $Y$ of $\bb$ and a two-level $3$-RIS $(y_i)_{i\in \N}$
from $Y$.
By the definition of two-level RIS for every $i\in \N$ there exists
a normalised $3$-RIS $y_{i,1},\dots,y_{i,j_i} \in Y$ with growth index
$j_i<j_{i,1}<\dots<j_{i,n_{j_i}}$ such that
$y_i = \frac{c_im_{j_i}}{n_{j_i}}\sum_{j=1}^{n_{j_i}} y_{i,j}$, where $c_i>0$ is some normalising
constant with $1/30 \leq c_i \leq 1$.
Moreover, we assume that $j_1\geq 5$.

Concerning basis averages we define for every $i\in \N$ a vector
$z_i = \frac{m_{j_i}}{n_{j_i}}\sum_{j=1}^{n_{j_i}} z_{i,j}$, where $z_{i,j}$'s with lexicographical
ordering on pairs $(i,j)$ is a subsequence of the standard basis.
Recall $\|z_i\|=1$ by Proposition \ref{baver-norm}.
We will show that for every finite sequence of scalars $(a_i)_{i\in I}$ we have
\[
    2^{-7}\|\sum_{i\in I} a_iz_i\| \leq \|\sum_{i\in I} a_iy_i\| \leq 
    2^{10}\|\sum_{i\in I} a_iz_i\|.
\]

\begin{lemma}
For every finite sequence of scalars $(a_i)_{i\in I}$ we have
\[
    \|\sum_{i\in I} a_iz_i\| \leq  2^7\|\sum_{i\in I} a_iy_i\|.
\]
\end{lemma}

\begin{proof}
Fix a finite sequence of scalars $(a_i)_{i\in I}$ and set $y=\sum_{i\in I} a_iy_i$, $z=\sum_{i\in I} a_iz_i$. 
We assume $\|z\|=1$ and $a_i>0$ for all $i\in I$ (by unconditionality of the standard basis of $\mT$).
Choose $f\in W$ with $f(z)=\|z\|$ and 
$\supp f\subseteq \supp z$.
Let $(f_t)_{t\in\T}$ be a tree analysis of $f$.
Consider $P=\{(i,j) \mid i\in\N,\,j=1,\dots,n_{j_i})\}$ with lexicographical ordering as a sequence.
For every $p\in P$ let $t_p\in\T$ be such that $f_{t_p} = \pm z^*_p$ (recall 
$(z_p)_{p\in P}$ is a subsequence of the basis).
Without loss of generality we can assume that $f_{t_p} = z^*_p$ for all $p\in P$.

Define $P'=\{p\in P\mid f(z_p) > 4^{-j_p-1}\}$. 
By Lemma \ref{lenpath} for all $p\in P'$ the branch linking the root and $t_p$ is of length less
than or equal to $j_p$.
Observe that $f(\sum_{p\in P\setminus P'}z_p) \leq 4^{-j_1}$, hence by relaxing $f(z)=1$, to 
$f(z) \geq 1/2$ we may assume that $P=P'$.
Let $\T'\subseteq\T$ be the minimal tree containing the root and all $t_p$'s for $p\in P$.

We will inductively construct a node $\gamma$ with the tree analysis 
$(I_t, \varepsilon_t, \eta_t)_{t\in\T'}$ satisfying for every  $t\in\T'$ the following
\[ 
    f_t(z) \leq 2^6 \varepsilon_te^*_{\eta_t}P_{I_t}(y),
\]
and for every non-terminal $t\in \T'$, writing $\bdp(e^*_{\eta_t})=\sum_{s\in S_t\cap\T'}d^*_{\beta_s}$,
and for every $s\in S_t\cap\T'$ we have 
\[
    \max\rng y_{p_s} < \rank(\beta_s) \leq \max\rng y_{p_s} + \len b(s, t_{p_s}),
\]
where $p_s=\max\{p\in P\mid s\preceq t_p\}$ and $\len b(s,t_{p_s})$ is the length of
the branch linking $s$ and $t_{p_s}$.
The induction on $\T'$ starts from terminal nodes.

Let $t$ be a terminal node in $\T'$. 
Then $t=t_p$ for some $p\in P$.
Choose $\eta_{t_p}\in \Gamma$ such that $|e^*_{\eta_{t_p}}(y_p)| \geq 1/2$ and 
$\rank(\eta_{t_p})\in \rng y_p$.
Then choose $\varepsilon_{t_p}$ such that 
$|e^*_{\eta_{t_p}}(y_p)| = \varepsilon_{t_p}e^*_{\eta_{t_p}}(y_p)$ and 
$I_{t_p}=\rng e^*_{\eta_{t_p}} \cap [\min\rng y_p, \infty)$.
This choice guarantees
\[ 
    f_{t_p}(z_p) = 1 \leq 2 \varepsilon_{t_p}e^*_{\eta_{t_p}}P_{I_{t_p}}(y_p).
\]
Using the lower bound $c_i \geq 1/30$ we get
\[ 
    f_{t_p}(z) \leq 2^6  \varepsilon_{t_p}e^*_{\eta_{t_p}}P_{I_{t_p}}(y).
\]

Let $t$ be a non-terminal node in $\T'$.
Then $f_t=m_t^{-1}\sum_{s\in S_t\cap\T'}f_s$ and we choose 
$\eta_t\in\Gamma$ such that 
\[
    \mt(e^*_{\eta_t})=m_t^{-1}\sum_{s\in S_t\cap\T'}\varepsilon_se^*_{\eta_s}P_{I_s},\ 
    \bdp(e^*_{\eta_t})=\sum_{s\in S_t\cap\T'}d^*_{\beta_s},
\]
with 
\[
    \rank(\beta_s) = 
    \begin{cases}
        \max\rng y_{p_s} + 1    & \text{, if $s$ is a terminal node,} \\
        \rank(\eta_s) + 1       & \text{, if $s$ is a non-terminal node.}
    \end{cases}
\]
Set $\varepsilon_t=1$ and $I_t=\rng(e^*_{\eta_t})$.

The choice of $\beta_s$'s, the bound of the length of the branch linking $z_p$'s and the root,
and the assumption on the size of gaps between consecutive $y_p$'s gives $\bdp(e^*_{\eta_t})(y)=0$.
Indeed, for every $s\in S_t\cap\T'$ we have $\len b(s, t_{p_s}) \leq j_{p_s}$ by Lemma \ref{lenpath} 
and $\max\rng y_{p_s} + j_{p_s} < \min\rng y_{p_s+}$ by the definition of two-level RIS.

Finally the inductive assumption gives
\[
    f_t(z) = m_t^{-1}\sum_{s\in S_t}f_s(z)  \leq \sum_{s\in S_t}d^*_{\beta_s}(y) + 
    m_t^{-1}\sum_{s\in S_t}2\varepsilon_se^*_{\eta_s}P_{I_s}(y) = 2^6 \varepsilon_te^*_{\eta_t}P_{I_t}(y).
\]
This finishes the inductive construction.
We set $\gamma=\eta_\emptyset$ and notice that 
\[
    \|z\|/2 = 1/2 < f_t(z) \leq  2^6 e^*_\gamma(y)\leq 2\|y\|.
\]
The proof of the lemma is finished.
\end{proof}

\begin{lemma}
For every finite sequence of scalars $(a_i)_{i\in I}$ we have
\[
    \|\sum_{i\in I} a_iy_i\| \leq  2^{10}\|\sum_{i\in I} a_iz_i\|.
\]
\end{lemma}

\begin{proof}
Fix a finite sequence of scalars $(a_i)_{i\in I}$ and set $y=\sum_{i\in I} a_iy_i$, $z=\sum_{i\in I} a_iz_i$. 
We assume $\|y\|=1$, hence for every $i\in I$ we have 
$|a_i| \leq \| P_{\rng y_i}(y)\| \leq 4$.
Let $\gamma$ be such that $|e^*_\gamma(y)|\geq 6/7$ and take its tree analysis $(I_t, \varepsilon_t, \eta_t)_{t\in\T}$.
Consider $P=\{(i,j) \mid i\in\N,\,j=1,\dots,n_{j_i})\}$ with lexicographical ordering as a sequence.
Notice that the sequence $(y_p)_{p\in P}$ is also a RIS.
Thus by applying Lemma \ref{comp} for $\gamma$ and $(a_iy_{i,j})_{(i,j)\in P}$ we can assume that the tree analysis of $\gamma$ is comparable with $(y'_{i,j})$, where $y'_p$ is a projection of $y_p$ on some interval, and $|e^*_\gamma(\sum_{(i,j)\in P'}a_iy'_{i,j})|\geq 1/7 - 27\cdot 4^{j_1}$, for some $P'\subseteq P$.
Since by $1$-unconditionality of the standard basis of $\mT$ we have $\|\sum_{(i,j)\in P'} a_iz_{i,j}\| \leq \|z\|$, we can assume $P=P'$.
Moreover, with the abuse of the notation we shall write $y_p, y_i$ also for the restrictions $y'_p$ and their sums.

For every $i$ we define $t_i\in\T$ to be the covering index for $y_i$.

Now we will show that there exists $f\in W$ such that $\supp f \subseteq \supp z$ and
\[ 
    |e^*_\gamma(y)| \leq 4n_{j_1}^{-1} + 96f(z). 
\]

Let $\T'$ be the smallest subtree of $\T$ containing the root
and all $t_i$'s and define for all $t\in\T'$ a tree 
$\T'_t = \{s \in \T'\mid s \succeq t\}$ and
\[
    A_t = \{ i \in I \mid \rng(e^*_{\eta_t}P_{I_t}) \cap \rng y_i \neq \emptyset\}, \ 
    y_t = \sum_{i\in A_t} a_iy_i,\ z_t=\sum_{i\in A_t} a_iz_i.
\]

\begin{claim}
For all $t \in \T'$ there is $f_t\in W$ such that $\supp f_t\subseteq \supp z_t$ and
\[
    |e^*_{\eta_t}P_{I_t}(y_t)| \leq |\bdp|(e^*_{\eta_t}, \T'_t)(y_t) + 96f_t(z_t).
\]
\end{claim}

\begin{proof}[Proof of the claim]
Proof by induction on $\T'$ starting from the terminal nodes of $\T'$.
The argument for the base of the induction is a simpler version of the argument
for the inductive step (see Case 2 below), hence we show only the latter.

Let $t\in \T'$ be a non-terminal node.
Consider the evaluation analysis of $e^*_{\eta_t}$
\[
    e^*_{\eta_t}=\sum_{s\in S_t} d^*_{\beta_s} + 
    \frac{1}{m_t}\sum_{s\in S_t} \varepsilon_s e^*_{\eta_s} P_{I_s}.
\]

\noindent Case 1. $t \neq t_i$ for all $i$.
Let $f_t = \frac{1}{m_t}\sum_{s\in S_t \cap \T'} f_s$.
Then using the inductive assumption we obtain
\begin{align*}
    |e^*_{\eta_t}P_{I_t}(y_t)| & \leq \sum_{s\in S_t} |d^*_{\beta_s}P_{I_t}(y_t)|
    + \frac{1}{m_t}\sum_{s\in S_t \cap \T'} |e^*_{\eta_s} P_{I_s} (y_s)| \\
    & \leq \sum_{s\in S_t} |d^*_{\beta_s}P_{I_t}(y_t)| + \frac{1}{m_t}\sum_{s\in S_t \cap \T'}|\bdp|(e^*_{\eta_s}, \T'_s)(y_t) + 
    \frac{1}{m_t}\sum_{s\in S_t \cap \T'} 96f_s(z_s) \\
    & \leq |\bdp|(e^*_{\eta_t}, \T'_t)(y_t) + 96f_t(z_t)
\end{align*}

\noindent Case 2. $t=t_{i_0}$ for some $i_0$.
We define $E_t=\{i\in I\mid t =t_i\}$ and estimate
\begin{align*}
    |e^*_{\eta_t}P_{I_t}(y_t)| & \leq \sum_{s\in S_t} |d^*_{\beta_s}P_{I_t}(y_t)|
    + \frac{1}{m_t}\sum_{s\in S_t \cap \T'} |e^*_{\eta_s} P_{I_s} (y_s)| + \frac{1}{m_t}\sum_{s\in S_t\setminus \T'} |e^*_{\eta_s} P_{I_s} (\sum_{i\in E_t} a_iy_i)|\\
    & \leq |\bdp|(e^*_{\eta_t}, \T'_t)(y_t) + 
    \frac{1}{m_t}\sum_{s\in S_t \cap \T'}96f_s(z_s) + \frac{1}{m_t}\sum_{s\in S_t\setminus  \T'} |e^*_{\eta_s} P_{I_s} (\sum_{i\in E_t} a_iy_i)|\\
    & \leq \dots
\end{align*}

Fix $i\in E_t$. 
Define $S_i=\{s\in S_t\setminus \T'\mid \rng e^*_{\eta_s}P_{I_s} \cap \rng y_i \neq \emptyset\}$ 
and $S_{i,j}=\{s\in S_t\setminus \T'\mid \rng e^*_{\eta_s}P_{I_s} \cap \rng y_{i,j} \neq \emptyset\}$. 
%
Let $A_i=\{s\in S_i\mid \exists j\colon \rng y_{i,j} \subset \rng (e^*_{\eta_s}P_{I_s}) \}$ 
and $B_i=S_i\setminus A_i$.
For every $s\in A$ let $J_s =\{j\mid \rng y_{i,j} \subset \rng (e^*_{\eta_{s}}P_{I_s})\}$.
Let $B_j = \{s\in B\mid \rng (e^*_{\eta_s}P_{I_{s}}) \subseteq \rng y_{i,j}\}$ 
for $j$ in $J_i=\{j\mid \exists s\in S_i\colon \rng (e^*_{\eta_s}P_{I_s}) \subseteq \rng y_{i,j}\}$.
Using the above notation we have the following splitting
\[
    \sum_{s\in S_i} |e^*_{\eta_s} P_{I_s} (\sum_{j=1}^{n_{k_i}} y_{i,j})| = \sum_{s\in A_i} |e^*_{\eta_s} P_{I_s} (\sum_{j\in J_s} y_{i,j})| + \sum_{j\in J_i} \sum_{s\in B_j} |e^*_{\eta_s} P_{I_s} (y_{i,j})|.
\]
Concerning the first part, by Lemma \ref{norm-sum} for every $s\in A$ there exists a functional $f_s\in W$ with $\supp f_s \subset \supp \sum_{j\in J_s} z_{i,j}$ and 
\[
    |e^*_{\eta_s} P_{I_s} (\sum_{j\in J_s} y_{i,j})| \leq 96 f_s(\sum_{j\in J_s} z_{i,j}).
\]
Concerning the second part, by Lemma \ref{norm-proj}  (see Remark \ref{aveproj})  for every $j\in J_i$ the following inequality holds:
\[
    \sum_{s\in B_j} |e^*_{\eta_s} P_{I_s} (y_{i,j})| \leq  9\left(1 + \frac{2\#B_j}{n_{k_{i,j}}}\right) \leq 18 =
    18f_j(z_{i,j}),
\]
for the functional $f_j\in W$ dual to $z_{i,j}$ as $\#B_j \leq |S_t| \leq n_{k_{i}} < n_{k_{i,j}}$.

Finally, we have 
\begin{align*}
    \sum_{s\in S_i} |e^*_{\eta_s} P_{I_s} (y_{i})| &= \sum_{s\in A_i} |e^*_{\eta_s} P_{I_s} (\frac{c_im_{k_i}}{n_{k_i}}\sum_{j\in J_s} y_{i,j})| + \sum_{j\in J_i} \sum_{s\in B_j} |e^*_{\eta_s} P_{I_s} (\frac{c_im_{k_i}}{n_{k_i}}y_{i,j})| \\
     & \leq 96\sum_{s\in A_i}f_s(\frac{m_{k_i}}{n_{k_i}}\sum_{j\in J_s} z_{i,j}) + 18\sum_{j\in J_i} f_{s_j}(\frac{m_{k_i}}{n_{k_i}} z_{i,j}).
\end{align*}

We define 
\[
    f_t=\frac{1}{m_t}\left(\sum_{s\in S_t\cap\T'}f_s + 
    \sum_{i\in E_t}(\sum_{s\in A_i} f_s + \sum_{j\in J_i} f_j)\right).
\]
The functional $f_t$ is in $W$ since all $f_s$'s and $f_j$'s have pairwise disjoint ranges and we have
$\sum_{i\in E_t}(\# A_i + \# J_i) \leq \#(S_t\setminus \T')$. 

Going back to the main estimation we obtain
\[
\dots \leq |\bdp|(e^*_{\eta_t}, \T_t')(y_t) + 96f_t(z_t),
\]
which finishes the inductive construction.
\end{proof} 

Going back to the main proof we have that $y_\emptyset=y,\ z_\emptyset = z,\ \eta_\emptyset=\gamma$, and
$|e^*_{\gamma}(y)| \leq |\bdp|(e^*_\gamma, \T')(y) + 96\|z\|$.
For every $i\in I$ we have $|a_i| \leq \|P_{\rng y_i}(y)\| \leq 4$,
and thus by Lemma \ref{est-bd} we have $|\bdp|(e^*_\gamma, \T')(y)\leq 4n_{j_1}^{-1}$.

The assumption on $\gamma$ yields
\[
    \frac{1}{7} - 27 \cdot 4^{-j_1} \leq \frac{4}{n_{j_1}} + 96 \|z\|,
\]
hence, as $j_1\geq 5$, we get
\[
    \|y\| \leq 8\cdot 96\|z\| \leq 2^{10} \|z\|.
\]
\end{proof}

\end{proof}

\end{document}